\documentclass{amsart}
\usepackage[utf8x]{inputenc}

\usepackage{amsthm,amsmath,amssymb,amstext,amsfonts,bm}

\usepackage[top=3cm,bottom=3cm,right=3cm,left=3cm]{geometry}
\usepackage{todonotes}
\usepackage[capitalize]{cleveref}

\newcommand{\field}[1]{\mathbb{#1}}

\newcommand{\Z}{\field{Z}}
\newcommand{\R}{\field{R}}

\newcommand{\esper}{\mathbb{E}}

\def\O{\mathcal{O}}
\def\a{\alpha}

\def\si{\sigma}
\def\Dep{L}
\def\WDep{G}

\def\eps{\varepsilon}

\def\bC{\bm{C}}

\def\ka{\kappa}

\def\bC{\bm{C}}

\def\One{\bm{1}}
\def\GP{\mathcal{\tilde{P}}}
\DeclareMathOperator\wt{wt}
\DeclareMathOperator\diam{diam}
\DeclareMathOperator\Int{Int}
\DeclareMathOperator\Var{Var}
\DeclareMathOperator\Cov{Cov}
\newcommand{\MWST}[1]{\mathcal{M} \big( #1 \big)}
\newcommand{\dist}[2]{\|#1 - #2\|_1}

\def\EEE{\mathcal{E}}

\numberwithin{equation}{section}
\newtheorem{theorem}{Theorem}[section]
\newtheorem{lemma}[theorem]{Lemma}
\newtheorem{corollary}[theorem]{Corollary}

\newtheorem{proposition}[theorem]{Proposition}
\theoremstyle{remark}
\newtheorem{remark}[theorem]{Remark}

\newtheorem{definition}[theorem]{Definition}

\renewenvironment{proof}[1][Proof]{\begin{trivlist}
\item[\hskip \labelsep {\bfseries #1:}]}{\qed\end{trivlist}}

\title{Weighted dependency graphs and the Ising model}
\author{Jehanne Dousse and Valentin F\'eray}
\address{Institut für Mathematik, Universität Zürich, Winterthurerstr. 190, CH-8032 Zürich, Switzerland}
\email{jehanne.dousse@math.uzh.ch,\ valentin.feray@math.uzh.ch}

\thanks{Both authors are partially supported by grant SNF-149461 from Swiss National Science Fundations}

\begin{document}

\begin{abstract}
  Weighted dependency graphs have been recently introduced by the second author,
  as a toolbox to prove central limit theorems.
  In this paper, we prove that spins in the $d$-dimensional Ising model 
  display such a weighted dependency structure.
  We use this to obtain various central limit theorems
  for the number of occurrences of local and global patterns
  in a growing box.
\end{abstract}

\maketitle

\section{Introduction and statement of results}
\subsection{Cumulants in the Ising model}
The Ising model is a mathematical model of ferro-magnetism in statistical physics. It was introduced in 1920 by Wilhelm Lenz who gave it as a problem to his Ph.D. student Ernst Ising~\cite{Ising}. It can be defined on any finite graph, but we restrict ourselves to finite subsets $\Lambda$ of $\Z^d$.
For any lattice site $i \in \Lambda$, there is a random variable $\sigma_i$ which is equal to either $1$ or $-1$ and represents the {\em spin} at site $i$. A \emph{spin configuration} $\omega = (\sigma_i (\omega))_{i \in \Lambda}$ is an assignment of spins to every site of $\Lambda$.

The distribution of spins depends on the magnetic field $h$ and the inverse temperature $\beta$ in a way that will be detailed later (see Section~\ref{sec:prelimIsing}).
In particular, spins corresponding to neighbour sites $i$ and $j$ are more likely to be aligned (i.e. both equal to $1$ or both equal to $-1$);
the bigger $\beta$ is (i.e. the lower the temperature is), the more important is this phenomenon.

In his Ph.D. thesis~\cite{Ising}, Ising solved the model for the one-dimensional case $d=1$, and showed that there is no phase transition. But in 1936, Peierls~\cite{Peierls} showed that, in dimensions $2$ and $3$, when $h=0$, the Ising model undergoes a phase transition at a critical inverse temperature $\beta_c$. He used a combinatorial argument now known as Peierls' argument. The two-dimensional model for $h=0$ was then exactly solved by Onsager~\cite{Onsager} in 1944, using analytic techniques and the transfer matrix method. It turns out that in higher dimensions, there is also a phase transition for $h=0$ (see \cite{Peierls} or
\cite{Velenik} for a more modern treatment)
and no phase transition when there is a magnetic field $h \neq 0$~\cite{LeeYang,Yang}.

The Ising model is {\em a priori} defined on a finite subset $\Lambda \subset \Z^d$,
but it is well-known that we can take the {\em thermodynamic limit} $\Lambda \uparrow \Z^d$ (see eg.~\cite{Velenik}).

This defines, for each pair of parameters $(\beta,h)$, a measure $\mu_{\beta,h}$
on the set $\{-1,1\}^{\Z^d}$ of spin configurations on the whole $d$-dimensional lattice $\Z^d$.
(In low temperature without magnetic field, i.e. $\beta$ large and $h=0$,
this measure is not unique; we will consider the one corresponding to $+$ boundary conditions,
see \cref{sec:prelimIsing} for details.)

The Ising model has since been studied in thousands of research articles, under various aspects.
Among many others, a subject of interest has been the decay of joint cumulants of the spins
(also called \emph{truncated $k$-point functions} or \emph{Ursell functions} in the physics literature).
The joint cumulant of order $r$ of spins $\sigma_{i_1},\dots,\sigma_{i_r}$ 
is defined as
$$ \kappa_{\beta,h}(\sigma_{i_1},\dots,\sigma_{i_r}) = [t_1 \dots t_r] \log 
\big\langle \exp(t_1 \sigma_{i_1} + \dots + t_r \sigma_{i_r}) \big\rangle_{\beta,h}.
$$
The covariance is the joint cumulant of order $2$ (more detail in Section~\ref{sec:cumulants}).

Bounds on cumulants in the physics literature are often called {\em cluster properties}.
There is in fact a hierarchy of cluster properties (corresponding to sharper or weaker
bounds on cumulants); we refer to \cite{Duneau1} or \cite[Chapter 6, §1]{Malyshev}
for definitions of various kinds of cluster properties.

In the case of the Ising model, a first bound on cumulants was obtained by Martin-Löf \cite[Eq. (20)]{Lof} --- see also \cite[Section 1]{Lebowitz} ---:
he proved that the joint cumulant $\kappa_{\beta,h}(\sigma_i;i \in A)$ decreases exponentially in $\diam(A)/r$, where $\diam(A)$ is the the diameter of $A$
and $r$ the order of the cumulants.
In~\cite{Duneau}, Duneau, Iagolnitzer and Souillard sharpened this bound in presence of a magnetic field ($h \ne 0$),
or for $h=0$ and very high temperature: $\kappa_{\beta,h}(\sigma_i;i \in A)$ decays exponentially in $\ell_T(A)$,
the minimum length of a tree connecting vertices of $A$ (see \cref{Subsect:Spanning} for a formal definition of the tree-length). 
In~\cite{Malyshev}, Malyshev and Minlos have a similar result in the case $h=0$ and very low temperature.
Both their approaches use cluster expansion, a powerful tool 
introduced by Mayer and Montroll~\cite{Mayer} which consists in viewing
our model in terms of macroscopic geometrical objects instead of considering its original microscopic components.
Both proofs use additional ingredients of different nature:
Duneau, Iagolnitzer and Souillard use Lee-Yang circle theorem and complex analysis arguments,
while Malyshev's and Minlos' approach relies on combinatorial developments and bounds on joint cumulants for \emph{contours}
as an intermediate step (contours will be defined in Section~\ref{Sect:CE_LT}).

These bounds on free cumulants will be our starting point to prove central limit theorems for patterns
in the Ising model.
In order to make the article more self-contained,
we give a simpler and more unified approach 
of the decays of joint cumulants in regimes, where the cluster expansion converges. 
The result is stated as follows.

\begin{theorem}
  We consider the Ising model on $\Z^d$ with parameters $(\beta,h)$.
  There exist positive constants $\eps(d)<1$, $\beta_1(d)$, $\beta_2(d)$ and $h_1(d)$ depending on the dimension $d$
  with the following property.
  Assume we are in one of three following regimes:
  \begin{description}
    \item[very low temperature] $\beta> \beta_2(d)$ and $h=0$;
    \item[very high temperature] $\beta < \beta_1(d)$ and $h=0$;
    \item[strong magnetic field] $h > h_1(d)$.
  \end{description}
  Then for any $r \ge 1$, there exists a constant $D_r$ such that for all $A=\{i_1, \dots, i_r\} \subset \Z^d$, we have
  \[\big| \kappa_{\beta,h}(\sigma_{i_1},\dots, \sigma_{i_r}) \big| 
  \le D_r \eps(d)^{\ell_T(A)}.\]
  \label{thm:bound_joint_cumulants}
\end{theorem}

That the tree-length appears as exponent
is important to make the connection with weighted dependency graphs, which we discuss now.

\subsection{Weighted dependency graphs}
\label{ssec:introWDG}
The theory of weighted dependency graphs, recently introduced in \cite{Valentin}, is a toolbox to prove central limit theorems.
It extends the well-known concept of dependency graphs; see \cite{Baldi_Rinott:DepGraphs_Stein,JansonDependencyGraphs}.

Throughout the article, a weighted graph is a graph such that a weight $w_e$ in $[0,1]$ is associated with each edge $e$,
where a weight $0$ is the same as no edge.
Informally, that a family of random variables $\{Y_a, a \in A\}$ admits a weighted graph $\WDep$ as weighted dependency graph means the following:
\begin{itemize}
  \item $G$ has vertex-set $A$, i.e. we have one vertex in $G$ per variable in $\{Y_a, a \in A\}$;
  \item the smaller the weight of an edge $\{a,b\}$ is,
    the {\em closer to independent} $Y_a$ and $Y_b$ should be.
    In particular, an edge of weight $0$, or equivalently no edge between $a$ and $b$,
    means that $Y_a$ and $Y_b$ are independent.
\end{itemize}
Formally, this closeness to independence is not only measured by a bound on the covariance (as could be expected),
but also involves bounds on higher order cumulants (see \cref{Def:WDG} for more precision).

Suppose now that, for each $n$, the 
family of random variables $\{Y_{a,n}, a \in A_n\}$ has a weighted dependency graph $G_n$.
Consider the renormalised sum $\widetilde{X_n}=\tfrac{1}{a_n}\left(\sum_{a \in A_n} Y_{a,n} \right)$.
Under some easy-to-check conditions on the renormalising factor $a_n$,
the variance of $\widetilde{X_n}$ and the maximal weighted degree of $G_n$,
$\widetilde{X_n}$ tends in distribution towards a Gaussian law (see \cref{th:4.11modif}).
In short, the theory of weighted dependency graph is a black box to prove central limit theorems.

A nice feature of weighted dependency graphs is the following stability property:
a weighted dependency graph for a family $\{Y_a, a \in A\}$
automatically gives a weighted dependency graph for monomials $Y_I=\prod_{a \in I} Y_a$
in the $Y_a$'s with a fixed bound on the degree (here, $I$ is a multiset of elements of $A$).
As a consequence, we can potentially prove central limit theorems for sums of such monomials.
We refer the reader to \cite{Valentin} for a detailed presentation of the theory
of weighted dependency graphs.
\medskip

Let us come back to the Ising model.
The bounds on joint cumulants of Theorem~\ref{thm:bound_joint_cumulants}
can be naturally translated in terms of weighted dependency graphs for the random variables $\{\sigma_i : i \in \Z^d \}$.
\begin{theorem}
\label{th:depgraph}
Let $\omega = (\sigma_i (\omega))_{i \in \Z^d}$ be a spin configuration
distributed according to $\mu_{\beta,h}$,
where either $h > h_1(d)$ or $(h=0;\, \beta < \beta_1(d))$ or $(h=0;\, \beta > \beta_2(d))$.
Let $\WDep$ be the complete weighted graph with vertex set $\Z^d$, such that every edge $e = (i,j)$ has weight $w_e = \eps(d)^{\frac{\dist{i}{j}}{2}}$, where $\eps(d)$ comes from Theorem~\ref{thm:bound_joint_cumulants}.

Then $\WDep$ is a $\mathbf{C}$-weighted dependency graph for the family $\{\sigma_i : i \in \Z^d \}$,
for some sequence  $\mathbf{C} = (C_r)_{r \geq 1}.$
\end{theorem}
\cref{th:depgraph} is proved in \cref{Sect:WDG_Spins}.
The proof uses \cref{thm:bound_joint_cumulants}, some general results of \cite{Valentin}
and elementary considerations.
As explained above,
this automatically yields a weighted dependency graph
for products of a finite number of spins,
which will be presented in Theorem~\ref{th:powers} below.
\bigskip

We conclude \cref{ssec:introWDG} with the motivation behind \cref{th:depgraph}.
The Ising model is the prototypical example of a Markov random field.
(Recall that a Markov random field on a graph $G$ with vertex set $A$ 
is a family of random variables
$\{Y_a, a \in A\}$ such that, for subsets $A_1$, $A_2$ and $A_3$,
$\{Y_a, a \in A_1\}$ and $\{Y_a, a \in A_2\}$ are independent conditionally on $\{Y_a, a \in A_3\}$
as soon as every path going from $A_1$ to $A_2$ in $G$ goes through $A_3$;
this is also sometimes called {\em global Markov property} \cite{MarkovRandomFields}).

Informally, in a Markov random field, a variable interacts directly only with its neighbours. 
We can thus expect that the dependency between variables is weaker when their distance in the graph $G$ increases
(since such variables only interact through all variables lying between them in the graph).
In other terms, we expect to have a weighted dependency graph that is complete (because there is no reason to have 
unconditionally independent variables), but whose weights decrease with the graph distance.
This was observed in the case of Markov chains (one-dimensional Markov random field)
in \cite[Section 10]{Valentin} and the present paper gives such a statement for the $d$-dimensional Ising model.
In both cases, weights decrease exponentially with the graph distance.

\subsection{Central limit theorems}
Central limit theorems (CLTs) play a key role in probability theory
and have also been a subject of interest in the study of the Ising model.
We refer to the second edition of Georgii's classical book
\cite[Bibliographic Notes on Section 8.2, p469]{Georgii2011}
for an overview of the different methods used to get such results.

The theory of weighted dependency graphs gives access to CLTs for the 
number of occurrences of {\em patterns of spins} in a growing box
$\Lambda_{n} := [-n,n]^d$.
To illustrate this, we consider two kinds of patterns: {\em local} and {\em global} patterns.

We define a \emph{local pattern} $\mathcal{P}$ to be a pair $(\mathcal{D}, \mathfrak{s})$, where $\mathcal{D}$ is a finite subset of $\Z^d$ containing $0$ and $\mathfrak{s}$ is a function $\mathcal{D} \longrightarrow \{+,-\}.$ The cardinality of $\mathcal{D}$ is called the \emph{size} of the pattern $\mathcal{P}.$
An example of local pattern is a positive spin surrounded by negative ones.
In that case the subset is \hbox{$\mathcal{D} = \{j \in \Z^d: \|j\|_1 \le 1\}$}, while the sign function is given by $\mathfrak{s}(0)=+$ and $\mathfrak{s}(j) = -$ for all $j \in \mathcal{D} \setminus \{0\}$. This pattern has size $2d+1$.
An \emph{occurrence} of a local pattern $\mathcal{P}=(\mathcal{D}, \mathfrak{s})$ is a set $\{(i+j,\mathfrak{s}(j)) : j \in \mathcal{D}\}$, where $i \in \Z^d$ is the \emph{position} of the occurrence. 

While in local patterns we consider spins that are at a fixed distance from one another, in global patterns they can be as far as we want, as long as they have a certain global shape. Formally, we define a \emph{global pattern} $\tilde{\mathcal{P}}$ of size $m$ to be a pair $(\mathcal{O}, \mathfrak{s})$, where $\mathcal{O}=(\leq_1, \dots , \leq_d)$ is a $d$-tuple of total orders over $\{1, \dots , m\}$, and $\mathfrak{s}$ is a function $\{1, \dots , m\} \longrightarrow \{+,-\}.$
An \emph{occurrence} of $\tilde{\mathcal{P}}$ in a spin configuration $\omega$ is a set $\{x^{(1)}, \dots ,x^{(m)}\}$ of $m$ elements of $\Z^d$ such that there exists some ordering $(x^{(1)}, \dots , x^{(m)})$ of these elements such that
\begin{enumerate}
\item for all $i \in \{1, \dots , m\}$, $\sigma_{x^{(i)}}(\omega) = \mathfrak{s}(i),$
\item for all $i,j \in \{1, \dots , m\}$, for all $k \in \{1, \dots d \}$, $x^{(i)}_k \leq x^{(j)}_k$ if and only if $i \leq_k j.$
\end{enumerate}
For example, if $d=2$, $\leq_1, \leq_2$ are both the natural ordering and $\mathfrak{s}(i) =+$ for all $i$,
then the global pattern $(\mathcal{O}, \mathfrak{s})$ is a North-East chain of $m$ positive spins.

CLTs for local and global patterns in other structures than the Ising model have attracted attention in the literature. We mention Markov chains (see \cite{Regnier_Spankowski:CLTPattern,FlajoletValleePatterns,Valentin} and references therein),
patterns in random permutations (see \cite{BonaMonotonePatterns,JansonNakamuraZeilberger} for global patterns and
 \cite{GoldsteinConsecutifs,BonaMonotonePatterns,ElizaldeEtAlConsecutifs} for local patterns)
and arc configurations in random set-partitions (CLTs for the number of arcs of size $1$, which is a local pattern,
and the number of crossings, which is a global pattern, were given in \cite{CLT_SetPartitionsStatistics}).
Note that Markov chains are (discrete) one-dimensional Markov random fields,
while the random permutation model is a non-Markovian two-dimensional model
(when considering patterns, we think of permutations as permutation matrices).
Finding such CLT results in Markov random fields of dimension two or more, and in particular in the Ising model,
is therefore a natural problem.
\bigskip

We first prove a CLT for local patterns.
Let $S_{n,\mathcal{P}}$ denote 
the number of occurrences of a given local pattern $\mathcal{P}$ in $\Lambda_n$.
\begin{theorem}
\label{th:cltpower}
Consider the Ising model on $\Z^d$, with inverse temperature $\beta$ and magnetic field $h$, such that either $h > h_1(d)$ or $(h=0;\, \beta < \beta_1(d))$ or $(h=0;\, \beta > \beta_2(d))$. Let $\mathcal{P}$ be a local pattern. Then
$$\frac{S_{n,\mathcal{P}} - \mathbb{E}(S_{n,\mathcal{P}}) }{\sqrt{|\Lambda_n|}} \xrightarrow[n\to\infty]{d} \mathcal{N}(0, v_{\mathcal{P}}^2).$$
\end{theorem}

Similarly, if $S_{n,\GP}$ denotes the number of occurrences of a global pattern $\GP$ in $\Lambda_n$,
we have the following result.
\begin{theorem}
  \label{thm:GlobalPatterns}
Consider the Ising model on $\Z^d$, with inverse temperature $\beta$ and magnetic field $h$, such that either $h > h_1(d)$ or $(h=0;\, \beta < \beta_1(d))$ or $(h=0;\, \beta > \beta_2(d))$. Let $\GP$ be a global pattern of size $m$.
We assume that, for some positive constants $A$ and $\eta$
\begin{equation}
  \Var(S_{n,\GP}) \ge A n^{2m-2+\eta}.
  \label{eq:Hypo_Variance}
\end{equation}
Then
\[\frac{S_{n,\GP} - \mathbb{E}(S_{n,\GP}) }{\sqrt{\Var(S_{n,\GP})}} \xrightarrow[n\to\infty]{d} \mathcal{N}(0, 1).\]
\end{theorem}
We do not have in general an estimate for the variance $\Var(S_{n,\GP})$.
However, when the pattern consists in positive spins only,
we can prove that \eqref{eq:Hypo_Variance} is satisfied (with $\eta=1$)
--- see \cref{prop:Variance_Lower_Bound} below.
The reverse inequality $\Var(S_{n,\GP}) \le B n^{2m-1}$ is always fulfilled
(see the proof of \cref{thm:GlobalPatterns}).
\medskip

We finish this introduction with a comparison with other methods.
Standard methods to get CLT in random fields are the use of mixing techniques \cite{neaderhouser1978}
or FKG inequalities \cite{Newman80,Newman83}.
It seems that the CLT for local patterns can be easily obtained with these methods.
Indeed, in an exponentially mixing field such as the Ising model,
if we denote $Z^{\mathcal P}_i$ the characteristic function of the occurrence
of $\mathcal P$ in position $i$ (see \eqref{eq:ZiP}),
then the field $(Z^{\mathcal P}_i)_{i \in \Z^d}$ is also exponentially mixing
and we can use the criterion given by Neaderhouser \cite[Section 3]{neaderhouser1978}.
CLT for functions of neighbouring spins are also accessible with methods based on FKG inequalities,
see \cite{Newman83} for a general result in this direction.
On the contrary, for global patterns, we do not know how to adapt these methods
to prove \cref{thm:GlobalPatterns}.
Indeed, if we denote $Z^{\GP}_{\{x_1,\dots,x_m\}}$ the characteristic function
of the occurrence of $\GP$ in position $\{x_1,\dots,x_m\}$ (see \eqref{eq:ZXGP}),
then $(Z^{\GP}_{\{x_1,\dots,x_m\}})$ is not a mixing field anymore.

The theory of weighted dependency graphs enables to deal with these different kinds of patterns in a
uniform way. In principle it would also be feasible to mix local and global conditions (as in vincular patterns for permutations \cite{Lisa});
besides the complexity of notation, a major difficulty is then to get
general estimates for the variance.

\subsection{Outline of the paper}
The remainder of the paper is organised as follows. In Section~\ref{sec:prelim}, we give some preliminary definitions and basic results about the Ising model, the theory of joint cumulants and weighted dependency graphs. In Section~\ref{sec:cluster}, we discuss the cluster expansion for the Ising model in the three different regimes we consider (high magnetic field, very high temperature, very low temperature) and deduce bounds on joint cumulants. In Section~\ref{sec:clt}, we use the theory of weighted dependency graphs to prove our central limit theorems.
\medskip

{\em Note:}
all constants throughout the paper depend on the dimension $d$ of the space and we shall not make it explicit from now on.

\section{Preliminaries}
\label{sec:prelim}
\subsection{The Ising model}
\label{sec:prelimIsing}
We consider the Ising model on a finite subset $\Lambda$ of $\Z^d.$ We use the notation of~\cite{Velenik}, that we define now.

Let $\EEE_{\Lambda} := \left\lbrace \{i,j\} \subset \Lambda : \dist{i}{j}=1 \right\rbrace$ be the set of nearest neighbour pairs in $\Lambda$.
(Here, and throughout the paper, $\dist{i}{j}$ denotes the graph distance in $\Z^d$ between two points $i$ and $j$.)
 To each spin configuration $\omega$, we associate its Hamiltonian
$$H_{\Lambda;\beta, h} (\omega):= -\beta \sum_{\{i,j\} \in \EEE_{\Lambda} } \sigma_i(\omega) \sigma_j(\omega) - h \sum_{i \in \Lambda} \sigma_i(\omega),$$
where $\beta \geq 0$ and $h$ are two real parameters,
respectively called {\em inverse temperature} and {\em magnetic field}.

The probability of a spin configuration $\omega$ is given by the Gibbs distribution
$$\mu_{\Lambda;\beta, h} (\omega):= \frac{1}{Z_{\Lambda;\beta, h}} e^{-H_{\Lambda;\beta, h} (\omega)},$$
where
$$Z_{\Lambda;\beta, h}:= \sum_{\omega \in \{-1,1\}^{\Lambda}} e^{-H_{\Lambda;\beta, h} (\omega)}$$
is called the partition function.

The quantities defined so far are with ``free boundary conditions'', which means that the value of the spins outside of $\Lambda$ is not taken into consideration. We can also define the same quantities with boundary condition, by considering the Ising model on the full lattice $\Z^d$, but where the values of the spins outside of $ \Lambda $ are fixed.
Fixing a spin configuration $\eta \in \{-1,1\}^{\Z^d}$, we define a spin configuration in $\Lambda$ with boundary condition $\eta$ as an element of the set
$$\Omega^{\eta}_{\Lambda} := \left\{ \omega \in \{-1,1\}^{\Z^d} : \omega_i = \eta_i, \forall i \notin \Lambda \right\}.$$
We now define the Hamiltonian as
$$H^{\eta}_{\Lambda;\beta, h} (\omega):= -\beta \sum_{\{i,j\} \in \EEE^{b}_{\Lambda} } \sigma_i(\omega) \sigma_j(\omega) - h \sum_{i \in \Lambda} \sigma_i(\omega),$$
where $\EEE^{b}_{\Lambda} := \left\lbrace \{i,j\} \subset \Z^d : \dist{i}{j}=1 \text{ and } \{i,j\} \cap \Lambda \neq \emptyset \right\rbrace .$

The Gibbs distribution of the Ising model in $\Lambda$ with boundary condition $\eta$ and parameters $\beta$ and $h$ is the probability distribution defined on $\Omega^{\eta}_{\Lambda}$ by
$$\mu^{\eta}_{\Lambda;\beta, h} (\omega):= \frac{1}{Z^{\eta}_{\Lambda;\beta, h}} e^{-H^{\eta}_{\Lambda;\beta, h} (\omega)},$$
where
$$Z^{\eta}_{\Lambda;\beta, h}:= \sum_{\omega \in \Omega^{\eta}_{\Lambda}} e^{-H^{\eta}_{\Lambda;\beta, h} (\omega)}$$
is the partition function with boundary condition $\eta$.

The most classical boundary conditions are the $+$ boundary condition, where $\eta_i = +1$ for all $i \in \Z^d$,
and the $-$ boundary condition, where $\eta_i = -1$ for all $i \in \Z^d$. When considering quantities with $+$ (resp. $-$) boundary condition, we write them with superscript $+$ (resp. $-$), e.g. $\mu^{+}_{\Lambda;\beta, h} (\omega).$

We now take an increasing sequence $\Lambda_n$ of finite subsets of $\Z^d$
with $\bigcup_{n \ge 1} \Lambda_n=\Z^d$.
It is a well-known fact (see, {\em e.g.}, \cite[Chapter 3]{Velenik})
that the sequence of measures $\mu^{+}_{\Lambda_n;\beta, h}$ converges in the weak sense
towards a measure denoted $\mu^+_{\beta, h}$ as $n \rightarrow \infty$.
In the high temperature case ($\beta < \beta_c(d),\, h=0$) or in the presence of a magnetic field 
($h \ne 0$), the limiting measure is independent of the choice of boundary conditions.
At low temperature ($\beta > \beta_c(d),\, h=0$), the limiting measure depends on the boundary conditions;
in this article, we restrict ourselves to $+$ boundary conditions
to have a well-defined limiting measure in all cases.
Also, we drop the superscript $+$ and denote the limiting measure by $\mu_{\beta,h}$.

In this article, we work with this limiting measure $\mu_{\beta,h}$
and prove our central limit theorem under this measure.
In comparison with the measure $\mu^{+}_{\Lambda_n;\beta, h}$, it has the advantage to be translation
invariant, which simplifies in particular the variance estimates.

\subsection{Joint cumulants}
\label{sec:cumulants}
As usual in the physics literature (and as done in the introduction), we use the notation $\langle f \rangle$
for the expectation of $f(X)$ (adding parameters of the measure as indices if necessary). 
For random variables $X_1,\dots,X_r$ on the same probability space with finite moments,
we define their {\em joint cumulant} (or {\em mixed cumulant}) as
\begin{equation}
    \kappa (X_1,\dots,X_r) = [t_1 \dots t_r] \log 
\big\langle \exp(t_1 X_1 + \dots + t_r X_r) \big\rangle.
    \label{EqDefCumulant}
\end{equation}
The notation $[t_1 \dots t_r] F$ stands here for the coefficient of $t_1 \dots t_r$ 
in the series expansion of $F$ in positive powers of $t_1, \dots, t_r$.
The finite moments assumption ensures that this series expansion exists.
If all random variables $X_1,\cdots,X_r$ are equal to the same variable $X$,
we denote $\kappa_r(X)=\kappa(X,\dots,X)$ and this is the usual {\em cumulant}
of a single random variable.
\medskip

Joint cumulants have a long history in statistics
and theoretical physics, see {\em e.g.} \cite{UrsellCumulants}.
In the case where the $X_i$'s are indicator functions of the presence of particles (or $+$ spins for examples),
they are often referred to in the statistical physics literature as {\em truncated correlation functions}
or {\em Ursell functions}.

\subsection{Spanning trees of maximal weight and  tree lengths}
\label{Subsect:Spanning}
\begin{definition}
A spanning tree of a graph $\Dep=(V,E)$ is
a subset $E'$ of $E$ such that $(V,E')$ is a tree.
\end{definition}
If $\WDep$ is an edge-weighted graph,
we define the weight $w(T)$ of a spanning tree of $\WDep$
as the {\em product} of the weights of the edges in $T$.
The maximum weight of a spanning tree of $\WDep$
is denoted $\MWST{\WDep}$.

We will be mainly interested in the case where $V$ is a finite subset $A$ of $\Z^d$,
$E$ consists of all pairs of vertices and
the weights are of the form $w(i,j)=\eps^{\dist{i}{j}}$ for some positive constant $\eps<1$.
We denote this weighted graph by $G_A$.
Then, for a spanning tree $T$ of $G_A$,
\[w(T) = \eps^{\sum_{(i,j) \in T} \dist{i}{j}},\]
and the maximal such weight  $\MWST{G_A}$ is obtained by {\em minimizing} the quantity $\sum_{(i,j) \in T} \dist{i}{j}$.
Therefore we define
\[\ell'_T(A) = \min_{T} \sum_{(i,j) \in T} \dist{i}{j}, \]
where the minimum is taken over all spanning trees $T$ of $G_A$, {\em i.e.} of the complete graph on $A$.
Then we have
$\MWST{G_A} = \eps^{\ell'_T(A)}$.

The quantity $\ell'_T(A)$ is sometimes referred to as the {\em tree-length} of $A$.
There is another closely related notion of tree-length defined
as $\ell_T(A)=\min_B \, \ell'_T(A \cup B)$, where the minimum is taken over all finite subsets $B$ of $\Z^d$.
In other words, this is the minimum length of a tree connecting vertices of $A$ and possibly other vertices of $\Z^d$.
Equivalently, this is the minimal size of a connected set of edges of the lattice $\Z^d$
such that each vertex of $A$ is incident to at least one edge in the set.
These two notions of tree-length are illustrated on \cref{fig:tree_length}.

\begin{figure}[ht]
\includegraphics[width=0.8\textwidth]{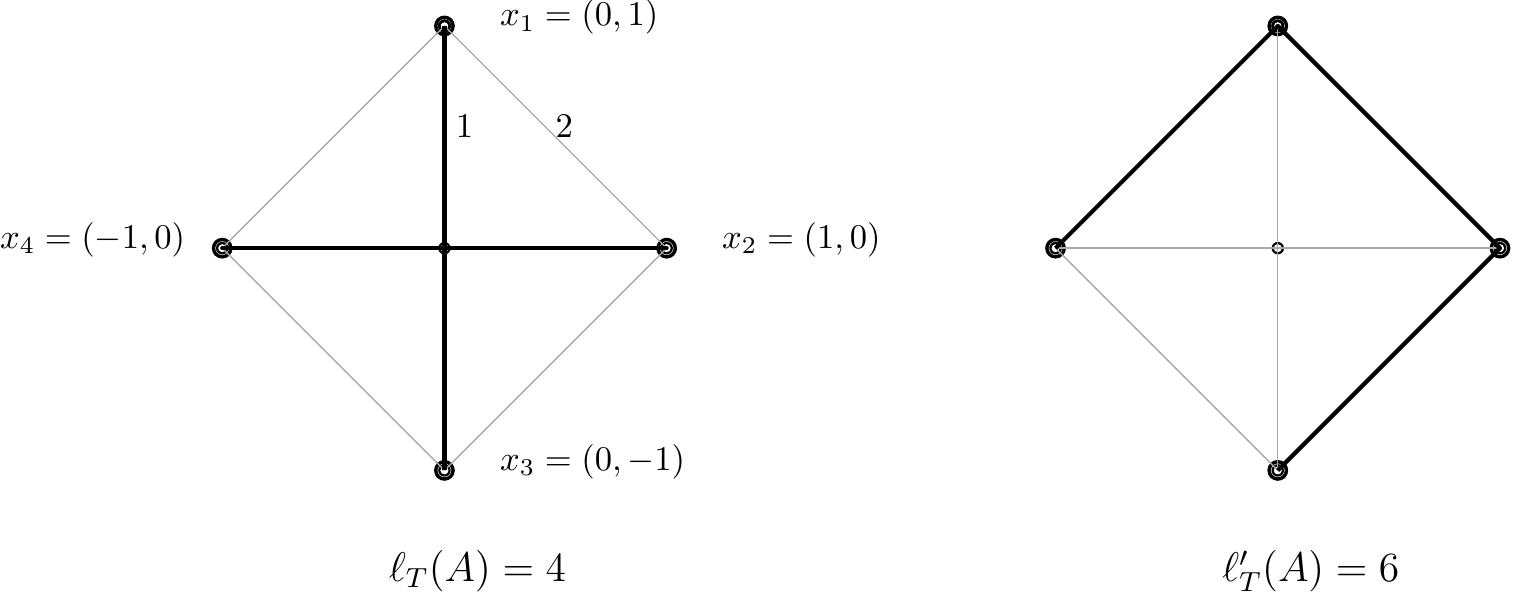}
\caption{The two tree lengths on an example}
\label{fig:tree_length}
\end{figure}

In~\cite[page 197]{Duneau1}, Duneau, Iagolnitzer and Souillard proved the following bound, 
which will be useful later in our computations.

\begin{proposition}
\label{prop:2longueurs}
For all $A=\{x_1,\dots,x_n\}$ finite subset of $\Z^d$, we have
$$\ell_T(A) \leq \ell'_T(A) \leq 2 \ell_T(A).$$
\end{proposition}

\subsection{Weighted dependency graphs}
Weighted dependency graphs have been introduced by the second author in \cite{Valentin}.
The following is a simplified definition, sufficient for the purpose of this paper
(it corresponds to the case $\Psi \equiv 1$ of the general definition, given in \cite{Valentin}).
\begin{definition}
\label{Def:WDG}
Let $\{Y_\a,\a \in A\}$ be a family of random variables with finite moments,
living in the same probability space;
and let $\bC=(C_1,C_2,\cdots)$ be a sequence of positive real numbers.

A weighted graph $\WDep$ is a {\em $\bC$-weighted dependency graph}
for $\{Y_\a,\a \in A\}$ if, 
for any multiset \hbox{$B=\{\a_1,\ldots,\a_r\}$} of elements of $A$,
one has
\begin{equation}
    \bigg| \ka\big( Y_\a ; \a \in B \big) \bigg| \le
    C_r \, \MWST{\WDep[B]}. 
    \label{EqFundamental}
\end{equation}
Here $\WDep[B]$ denotes the graph induced by $\WDep$ on the vertex set $B$.
\end{definition}

Weighted dependency graphs are a toolbox to prove central limit theorems.
Here is a normality criterion, which is a slightly
modified version of the main theorem in~\cite{Valentin}.

\begin{theorem}
\label{th:4.11modif}
Suppose that, for each $n$, $\{ Y_{n,i}, 1 \leq j \leq N_n\}$ is a family of random variables with finite moments defined on the same probability space. 
Let $\mathbf{C} = (C_r)_{r \geq 1}$ be a fixed sequence that does not depend on $n$.

Assume that, for each $n$, one has a $\mathbf{C}$-weighted dependency graph $\WDep_n$ for $\{ Y_{n,i}, 1 \leq j \leq N_n\}$ 
and denote $\Delta_n-1$ its maximal weighted degree.

Let $X_n = \sum_{i=1}^{N_n} Y_{n,i}$ and $v^2_n = Var (X_n)$.
Assume that there exists a sequence $(a_n)$, an integer $s \ge 3$ and a real number $v$ such that

\begin{enumerate}
\item $\displaystyle \frac{v_n^2}{a_n^2} \xrightarrow[n \to\infty]{} v^2,$
\item $\displaystyle \text{for all $n$, } a_n^2 \leq C_2 N_n \Delta_n,$
\item $\displaystyle \left( \frac{N_n}{\Delta_n} \right) ^{\frac{1}{s}} \frac{\Delta_n}{a_n} \xrightarrow[n\to\infty]{} 0.$
\end{enumerate}

Then in distribution,
$$\frac{X_n - \mathbb{E}(X_n)}{a_n} \xrightarrow[n\to\infty]{d} \mathcal{N}(0, v^2).$$
\end{theorem}

\begin{proof}
The proof is almost identical to the proof of the normality criterion in \cite[Section 4.3]{Valentin} replacing $\sigma_n$ by $a_n$.
Indeed, as noticed in \cite[Section 4.3]{Valentin}, in the special case $\Psi \equiv 1$ (to which we restrict ourselves in this article),
the quantities $R_n$ and $Q_n$ defined there can be replaced respectively by $N_n$ (the number of vertices)
and $\Delta_n$ (the maximal weighted degree plus one).
\end{proof}

\section{Cluster expansions and bounds on joint cumulants}
\label{sec:cluster}

The cluster expansion is a powerful tool in statistical mechanics, which consists in studying a system in terms of macroscopic geometrical objects instead of considering its original microscopic components. It was introduced in a work of Mayer and Montroll~\cite{Mayer} studying molecular distribution and has since been used in several other topics;
for the Ising model, see for example~\cite{Farrell} or Chapter $5$ of~\cite{Velenik}. In this section, we will use the cluster expansion in three different regimes of the Ising model to prove the bounds on joint cumulants of Theorem~\ref{thm:bound_joint_cumulants}. This will later be useful to apply the theory of weighted dependency graphs.
Theorem~\ref{thm:bound_joint_cumulants} is proved in \cref{SectBoundHT,SectBoundLT,SectBoundMF},
depending on the considered regime.

\begin{remark}
  As already mentioned, cluster expansion is a key step to obtain bounds for cumulants
  in each of the three regimes.
  Therefore we will use each time classical notation for cluster expansion,
  such as $\Xi$, $\wt$, \dots
  Note however that these quantities may have different meanings in different regimes.
  Since they are only used for the proof of \cref{thm:bound_joint_cumulants}
  and since the proofs in the different regimes are independent from each other,
  this should not create any difficulty.
\end{remark}

\subsection{At very high temperature, without magnetic field}
\label{Sect:CE_HT}

\subsubsection{The cluster expansion of the (multivariate) moment generating function}
Let us start with the regime where $h=0$ and $\beta$ is sufficiently small (very high temperature).

Fix a finite domain $\Lambda \subset \Z^d$ and let $A=\{x_1,\dots,x_r\}$ be a set of points in $\Lambda$.
We consider the (multivariate) moment generating function
\[ \left\langle \exp\left(\sum_{j=1}^r t_j \sigma_{x_j} \right) \right\rangle_{\Lambda;\beta,0} = 
\frac{\sum_{\omega \in \Omega_{\Lambda}} 
\exp\left(\sum_{j=1}^r t_j \sigma_{x_j}(\omega) \right) e^{-H_{\Lambda;\beta,0}(\omega)}}{Z_{\Lambda; \beta,0}}.\]

Let us call $Z^A_{\Lambda; \beta,0}$ the numerator of the right-hand side.
The denominator $Z_{\Lambda; \beta,0}$ is then exactly $Z^\emptyset_{\Lambda; \beta,0}$.
Let $\mathcal{E}^{\text{even}}_{A}$ (resp. $\mathcal{E}^{\text{even}}_{\Lambda,A}$) be the set of
pairs $(E,B)$, where $E \subset \EEE_{\Z^d}$ (resp. $E \subset \EEE_\Lambda$) and $B\subseteq A$ are such that
a vertex of $\Lambda$ is incident to an odd number of edges in $E$ if and only if it is in $B$.
For such a pair $(E,B)$, we denote $\wt(E,B)=(\tanh\, \beta)^{|E|} \prod_{j \in B} (\tanh\, t_j)$.

\begin{lemma}
  [high temperature representation]
We have
\begin{equation}
Z^A_{\Lambda; \beta,0} = 2^{|\Lambda|} (\cosh\, \beta)^{|\EEE_\Lambda|}
\left(\prod_{j=1}^r \cosh(t_j) \right) \Xi^A_{\Lambda;\beta,0},
\label{eq:PartFunction_HighTemperature_Modified}
\end{equation}
where $\Xi^A_{\Lambda;\beta,0}=\sum_{(E,B) \in \mathcal{E}^{\text{even}}_{\Lambda,A}} \wt(E,B)$.
  \label{lem:HTR}
\end{lemma}
\begin{proof}
  This proof is a straight-forward extension of the case $A=\emptyset$, see e.g.  \cite[Eq. (5.40)]{Velenik}.
  We write in short $\sigma_i$ for $\sigma_i(\omega)$.
  Since $\sigma_{x_j}$ is in $\{-1,+1\}$, we can write 
  \begin{align*}
    \exp(t_j \sigma_{x_j}) &= \cosh(t_j) + \sigma_{x_j} \sinh(t_j) 
    = \cosh(t_j) \big(1 + \sigma_{x_j} \tanh(t_j));\\
    \exp(\beta \sigma_i \sigma_j) &= \cosh(\beta) \big(1 + \sigma_i \sigma_j \tanh(\beta) \big).
  \end{align*}
  This gives the following expression for $Z^A_{\Lambda; \beta,0}$:
  \[
    Z^A_{\Lambda; \beta,0}= (\cosh\, \beta)^{|\EEE_\Lambda|}
  \left(\prod_{j=1}^r \cosh(t_j) \right) \sum_{\omega \in \Omega_{\Lambda}}
  \left[ \sum_{E \subseteq \EEE_\Lambda} \prod_{\{i,j\} \in E} (\sigma_i \sigma_j \tanh(\beta)) \right]\\
  \cdot\left[ \sum_{B \subseteq A} \prod_{x_j \in B} (\sigma_{x_j} \tanh(t_j)) \right].
  \]
  Changing the order of summation we should evaluate, for $E \subseteq \EEE_\Lambda$ and $B \subseteq A$, the quantity
  \[ \sum_{\omega \in \Omega_{\Lambda}}  \prod_{\{i,j\} \in E} 
  (\sigma_i \sigma_j) \prod_{x_j \in B} (\sigma_{x_j}). \]
  By an easy symmetry argument, this sum is zero unless all $\sigma_i$'s appear an even number of times, which corresponds to the condition $(E,B) \in \mathcal{E}^{\text{even}}_{\Lambda,A}$.
  In this latter case, the sum is the number of spin configurations $|\Omega_{\Lambda}|=2^{|\Lambda|}$.
  This ends the proof of the high temperature expansion.
\end{proof}

Pairs $(E,B)$ in  $\mathcal{E}^{\text{even}}_{\Lambda,A}$ can be considered as subgraphs of $\Lambda$,
where the vertices are all vertices incident to an edge of $E$
and the edge-set is precisely $E$ (vertices in $B$ must be adjacent to at least one edge in $E$).
This graph has a unique decomposition (up to reordering) into $r$ connected components,
each again being the graph of some $(E_i,B_i)$ (for $1 \le i \le r$).
Notice that the weight function is multiplicative with respect to connected components,
{\em i.e.} $\wt(E,B)=\prod_i \wt(E_i,B_i)$.
Therefore using notation of \cite{Velenik},
\[ \Xi^A_{\Lambda;\beta,0} = 1 + \sum_{r \ge 1} \frac{1}{r!} \sum_{(E_1,B_1), \dots, (E_r,B_r) \atop \text{connected}}
\ \prod_{i=1}^r \wt(E_i,B_i) \, \prod_{1 \le i<j \le r} \delta\big[(E_i,B_i),(E_j,B_j) \big], \]
where 
$\delta\big[(E_i,B_i),(E_j,B_j) \big]=1$ if the graphs corresponding to $(E_i,B_i)$
and $(E_j,B_j)$ do not share a vertex, and $0$ otherwise
(this factor encodes the fact that connected components should not intersect).
We set $\zeta((E,B),(E',B'))=\delta\big[(E,B),(E',B')\big]-1$
as usual for cluster expansions.

To compute cumulants, we need an expansion of 
$\log\left(\left\langle \exp\left(i \sum_{j=1}^r t_j \sigma_{x_j} \right) \right\rangle_{\Lambda;\beta,0} \right)$
and thus of $\log(\Xi^A_{\Lambda;\beta,0})$.
Such an expansion will be given by cluster expansion, but we should first check some conditions
ensuring convergence, {\em e.g.} the ones given in \cite[Section 5.4]{Velenik}.
We let $\overline{\wt}(E,B)=(\tanh\, \beta)^{|E|}$,
which dominates all functions $\wt(E,B)$ when the $t_j$'s are complex parameters
of moduli at most $\tanh^{-1}(1)$.
\begin{lemma}\label{lem:Cv_CE_HighTemperature}
For $(E,B) \in \mathcal{E}^{\text{even}}_{\Lambda,A}$, let $a(E;B)=|V(E,B)|$,
that is the number of vertices in the graph associated to $(E,B)$.
Then, there exists a constant $\beta^{\text{ht}}_{\text{ce}}(d)$ such that the following holds for $\beta<\beta^{\text{ht}}_{\text{ce}}(d)$:
\begin{enumerate}
\item for any $\Lambda \subset \Z$, we have $\displaystyle \qquad \sum_{(E,B) \in \mathcal{E}^{\text{even}}_{\Lambda,A}} |\overline{\wt}(E,B)| e^{a(E,B)}  < \infty ; $
\item \label{item:deux} for each fixed pair $(E_\star,B_\star)$ where $E_\star$ is a finite subset of $\mathcal{E}_{\Z^d}$
  and $B^\star \subseteq A$, one has
\[ S_{(E_\star,B_\star)}:=\sum_{(E,B) \in \mathcal{E}^{\text{even}}_{A} \atop (E,B) \text{ connected}} |\overline{\wt}(E,B)| e^{a(E,B)}
|\zeta\big[(E,B),(E_\star,B_\star)\big]| \le a(E_\star,B_\star).\]
\end{enumerate}
\end{lemma}
\begin{remark}
  To prove the convergence of cluster expansion, it is actually enough to prove a weaker version of Item \eqref{item:deux},
  where $(E_\star,B_\star)$ is necessarily in $\mathcal{E}^{\text{even}}_{\Lambda,A}$ and the sum only runs
  on $(E,B) \in \mathcal{E}^{\text{even}}_{\Lambda,A}$. 
  The stronger version stated here will nevertheless be useful in the proof of \cref{lem:BoudingTheSum_HT} below
  and is proved in the same way, which explains our choice.
\end{remark}
\begin{proof}
The first condition is trivial since the set $\mathcal{E}^\text{even}_{\Lambda,A}$ is finite.
Let us consider the second one.
By definition, $\zeta\big[(E,B),(E_\star,B_\star)\big]=-1$
if $(E,B)$ and $(E_\star,B_\star)$ share a vertex and $0$ otherwise.
Thus 
\[S_{(E_\star,B_\star)} \le \sum_{v \in V(E_\star,B_\star)} \left[\sum_{E \subset \EEE_{\Z^d} , B \subset \Z^d \text{ s.t. } (E,B) \text{ connected}
\atop \text{and } v \in V(E,B)}
|\overline{\wt}(E,B)|\, e^{|V(E,B)|}\right]. \]
A simple translation argument shows that the quantity between brackets is independent of $v$ so that
\[S_{(E_\star,B_\star)} \le  |V(E_\star,B_\star)| 
\left[\sum_{E \subset \EEE_{\Z^d}, B \subset \Z^d \text{ s.t. } (E,B) \text{ connected}
\atop \text{and } 0 \in V(E,B)}
|\overline{\wt}(E,B)|\, e^{|V(E,B)|}\right]. \]
Note that $(E,B)$ connected implies in particular that $B$ is included in the vertex set of the graph associated to $E$.
Therefore the sum can be simplified as a sum only over $E$ by paying a factor $2^{|V(E,B)|}$.
Moreover connectedness implies $|V(E,B)| \le |E|+1$.
Thus we get:
\[S_{(E_\star,B_\star)} \le  |V(E_\star,B_\star)| \left[\sum_{E \subset \EEE_{\Z^d} \text{ connected} \atop \text{s.t. } 0 \in V(E)}
2^{|E|+1} \, (\tanh\, \beta)^{|E|} \, e^{|E|+1} \right].\]
The summand depends only on the size $k$ of $E$.
From \cite[Lemma 3.59]{Velenik}, the number of connected sets $E \subset \EEE_{\Z^d}$ containing $0$
of size $k$ is bounded from above by $(2d)^{2k}$, so that
\[S_{(E_\star,B_\star)} \le |V(E_\star,B_\star)|  \left[ \sum_{k \ge 1} (2d)^{2k} \, 2^{k+1} \, (\tanh\, \beta)^{k}\, e^{k+1} \right].\]
For $\beta$ small enough, say $\beta < \beta^{\text{ht}}_{\text{ce}}(d)$, the sum is smaller than $1$ and the second inequality is fulfilled.
\end{proof}
We can now state the {\em cluster expansion} of $\log(\Xi^A_{\Lambda;\beta,0})$.
In this context, a cluster is a multiset \hbox{$X=\{(E_1,B_1),\dots,(E_r,B_r)\}$} of elements of $\mathcal{E}^{\text{even}}_{\Lambda,A}$.
The multiplicity of $(E,B)$ in the multiset $X$ is denoted $n_X(E,B)$.
The support $\overline{X}$ is the union of the vertex-sets $V(E_i,B_i)$.
We say that two clusters intersect if $\zeta\big[(E,B),(E_\star,B_\star)\big]=-1$ 
i.e. if they share a {\em vertex}.
\begin{proposition}
For $\beta<\beta^{\text{ht}}_{\text{ce}}(d)$, We have the following expansion:
\begin{equation}
 \log(\Xi^A_{\Lambda;\beta,0})
= \sum_{X: \overline{X} \subset \Lambda} \Psi_{\beta}(X),
\label{eq:CE_HighTemperature}
\end{equation}
where for a cluster $X= \{(E_1,B_1),\dots,(E_r,B_r)\}$,
$$\Psi_\beta(X):= \left( \prod_{(E,B) \in \mathcal{E}^\text{even}_{\Lambda,A}}
 \frac{1}{n_X(E,B)!} \right)  \, \phi((E_1,B_1),\dots,(E_r,B_r)) \, \left( \prod_{ i=1}^r \wt(E_i,B_i) \right),$$
where
$$\phi(( E_1,B_1),\dots,(E_r,B_r)) =\sum_{G \subset G_r\, \text{connected}} \left( \prod_{\{i,j\} \in G}
\zeta\big[(E_i,B_i),(E_j,B_j)\big] \right),$$
and $G_r$ denotes the complete graph on $r$ vertices.
The convergence of the series in \cref{eq:CE_HighTemperature} holds in the sense 
of locally uniform convergence of analytic functions in the complex parameters $t_1$, \dots, $t_r$
for $|t_1|,\dots,|t_r| \le \tanh^{-1}(1)$.
\label{prop:CE_HighTemperature}
\end{proposition}
\begin{proof}
This follows from the general theory of cluster expansions, see e.g. \cite[Chapter 5]{Velenik}.
For the analyticity in the parameters, see specifically \cite[Section 5.5]{Velenik}.
\end{proof}
Notice that the functional $\phi((E_1,B_1),\dots,(E_r,B_r))$ depends only on which pairs $(E_i,E_j)$ intersect
and vanishes if $X$ can be split into two mutually non-intersecting subsets.

\subsubsection{Bounds on joint cumulants}
\label{SectBoundHT}
Recall that $A=\{x_1,\dots,x_r\}$ is a set of points in the finite domain $\Lambda$.
The joint cumulant $\ka_{\Lambda;\beta,0}(\sigma_{x_1},\dots,\sigma_{x_r})$ is
 the coefficient of $t_1 \dots t_r$ in
\begin{multline*}
 \log \left\langle \exp\left(\sum_{j=1}^r t_j \sigma_{x_j} \right) \right\rangle_{\Lambda,\beta,0}
=\log Z^A_{\Lambda; \beta,0}  - \log Z^\emptyset_{\Lambda; \beta,0} \\
= |\Lambda|\, \log 2 + |\EEE_\Lambda| \, \log(\cosh\, \beta)
+\sum_{j=1}^r  \log (\cos\, t_j) + \log \Xi^A_{\Lambda;\beta,0}
-  \log Z^\emptyset_{\Lambda; \beta,0}.
\end{multline*}
Only the summand $\log \Xi^A_{\Lambda;\beta,0}$ contributes to the coefficient of $t_1 \dots t_r$.
Therefore, using \cref{prop:CE_HighTemperature}, we have
\begin{equation}
\label{eq:cumulants_cluster}
\ka_{\Lambda;\beta,0}(\sigma_{x_1},\dots,\sigma_{x_r}) = [t_1 \dots t_r] \sum_{X: \overline{X} \subset \Lambda} \Psi_{\beta}(X)
=  \sum_{X: \overline{X} \subset \Lambda}  [t_1 \dots t_r]  \Psi_{\beta}(X).
\end{equation}
The exchange of infinite sum and coefficient extraction is valid 
since we have uniform convergence of analytic functions on a neighborhood of $0$.
A cluster $X=\{(E_1,B_1),\dots,(E_r,B_r)\}$ contributes to the coefficient of $t_1 \dots t_r$
only if $A=B_1 \uplus \dots \uplus B_r \subseteq \overline{X}$.
Then 
\begin{multline*} [t_1 \dots t_r]  \Psi_{\beta}(X) = 
\left( \prod_{(E,B) \in \mathcal{E}^\text{even}_{\Lambda,A}}
 \frac{1}{n_X(E,B)!} \right)  \, \varphi((E_1,B_1),\dots,(E_r,B_r)) \, (\tanh \beta)^{e(X)}\\
 \le C_r \, \One[X \text{ is connected}] \, (\tanh \beta)^{e(X)},
 \end{multline*}
 where $e(X)=|E_1|+\dots+|E_r|$ and $\One[\text{event}]$ is the indicator function of the corresponding event.
 Back to \cref{eq:cumulants_cluster}, we get
 \[|\ka_{\Lambda;\beta,0}(\sigma_{x_1},\dots,\sigma_{x_r}) | \le
  \sum_{X: \overline{X} \subset \Lambda, X\text{ connected,}\atop \text{and } A \subseteq \overline{X}}
  C_r (\tanh \beta)^{e(X)}.\]
  Taking the limit $\Lambda \uparrow \Z^d$, we get a similar upper bound for the cumulant
  $\ka_{\beta,0}(\sigma_{x_1},\dots,\sigma_{x_r})$ under the probability measure 
  $\mu_{\beta,0}$ corresponding to the whole lattice $\Z^d$:
  \begin{equation}\label{eq:BoundCum1}
|\ka_{\beta,0}(\sigma_{x_1},\dots,\sigma_{x_r}) | \le
  \sum_{X\text{ connected}\atop \text{s.t. } A \subseteq \overline{X}}
  C_r (\tanh \beta)^{e(X)}.
  \end{equation}
  The key point in the above formula is that any connected cluster $X$
  with $A \subseteq \overline{X}$ fulfills $e(X) \ge \ell_T(A)$.
  We now need the following lemma, whose proof is inspired by the end of the proof
  of Theorem 5.27 in \cite{Velenik}.
 \begin{lemma}
 \label{lem:BoudingTheSum_HT}
There exist constants $ \beta^{\text{ht}}_{\text{jc}}(d)>0$ and $\eps>0$
such that, for $\beta \le \beta^{\text{ht}}_{\text{jc}}(d)$,
 we have the following inequality:
 \[\sum_{X\text{ connected}\atop \text{s.t. } x_1 \in \overline{X}\text{and } e(X) \ge R} (\tanh \beta)^{e(X)}
 \le \eps^R.\]
 \end{lemma}
 \begin{proof}
The proof involves different values of the inverse temperature $\beta$
so that we will here make explicit the dependency of the weight in $\beta$:
we write $\overline{\wt}_\beta(E,B)$  instead of $\overline{\wt}(E,B)$.
 We first prove that for $\beta' \le \beta^{\text{ht}}_{\text{jc}}(d)$, we have
 \begin{equation}\label{eq:SumPsiLeqOne}
\sum_{X\text{ connected}\atop \text{s.t. } x_1 \in \overline{X}} (\tanh \beta')^{e(X)} =\sum_{X\text{ connected}\atop \text{s.t. } x_1 \in \overline{X}}  \prod_{i=1}^r \overline{\wt}_{\beta'}(E_i,B_i)  \le 1.
 \end{equation}
 This uses the same argument as in \cite[Eq. (5.31)]{Velenik}:
 \begin{multline*}
 \sum_{X\text{ connected}\atop \text{s.t. } x_1 \in \overline{X}}  \prod_{i=1}^r \overline{\wt}_{\beta'}(E_i,B_i)
 \le \sum_{r \ge 1} r \sum_{(E_1,B_1) \text{ connected} \atop \text{s.t. } x_1 \in V(E_1,B_1)} \ \sum_{(E_2,B_2),\dots,(E_r,B_r)}
 \prod_{i=1}^r \overline{\wt}_{\beta'}(E_i,B_i) \\
 \le \sum_{(E_1,B_1) \text{ connected} \atop \text{s.t. } x_1 \in V(E_1,B_1)}  \overline{\wt}_{\beta'}(E_1,B_1) e^{|a(E_1,B_1)|}
 \le a(\emptyset, \{x_1\}) =1,
 \end{multline*}
 where we used \cref{lem:Cv_CE_HighTemperature} and \cite[Theorem 5.4]{Velenik}.
 This proves \eqref{eq:SumPsiLeqOne}.
 
 Let us fix a value $\beta'$ as above. 
 There exists a constant $\eps <1$ such that for $\beta$ small enough, we have $\tanh \beta < \eps\, \tanh \beta'$.
 We can now write
 $$\sum_{X\text{ connected}\atop \text{s.t. } x_1 \in \overline{X} \text{and } e(X) \ge R} (\tanh \beta)^{e(X)} 
 \le  \eps^R \sum_{X\text{ connected}\atop \text{s.t. } x_1 \in \overline{X} \text{and } e(X) \ge R} (\tanh \beta')^{e(X)} \le \eps^R,$$
 where the last inequality uses \eqref{eq:SumPsiLeqOne}. This ends the proof of the lemma.
 \end{proof}

Combining \cref{eq:BoundCum1,lem:BoudingTheSum_HT}, we get the desired bound:
for $\beta \le \beta^{\text{ht}}_{\text{jc}}(d)$,
\[ |\ka_{\beta,0}(\sigma_{x_1},\dots,\sigma_{x_r}) | \le C_r \eps^{\ell_T(A)}.\]

\medskip

\subsection{At very low temperature, without magnetic field}~ 
\label{Sect:CE_LT}

\subsubsection{The cluster expansion of the partition function}
We now turn to the regime without magnetic field ($h=0$) and very low temperature ($\beta$ large).
Intuitively, in that case, the spin configurations with fewer pairs of neighbours having opposite spins appear with higher probability. 
To emphasize the role of these pairs, we rewrite the Hamiltonian as follows:
$$H^{\eta}_{\Lambda;\beta, h} (\omega)= -\beta |E^{\eta}_{\Lambda}| -\beta \sum_{\{i,j\} \in \EEE^{\eta}_{\Lambda} } (\sigma_i(\omega) \sigma_j(\omega) - 1).$$
The only non-zero terms in the sum are those where two neighbours $i$ and $j$ have opposite spins.
Let us consider a finite subset $\Lambda \subset \Z^d$ with $+$ boundary condition. A typical spin configuration will then look as a sea of $+$'s with some islands of $-$'s. Therefore the interesting macroscopic components for the cluster expansion in that case are the frontiers between the areas of $+$'s and those of $-$'s, which are called \emph{contours}. Let us define them more rigorously.

Given $\omega \in \Omega_{\Lambda}^{+}$, let $\Lambda^{-}(\omega)$ denote the set of lattice points $i$ where $\sigma_i(\omega)=-1.$
For each $i \in \Z^d$ we define $\mathcal{S}_i := i + [\frac{-1}{2},\frac{1}{2}]^d$ to be the unit cube of $\R^d$ centred at $i$. Now let
$$\mathcal{M}(\omega) := \bigcup_{i \in \Lambda^{-}(\omega)} \mathcal{S}_i,$$
and consider the set of maximal connected components of the boundary of $\mathcal{M}(\omega)$, which we denote
$$\Gamma'(\omega)= \{ \gamma_1, \dots, \gamma_r \}.$$
Each of the $\gamma_i$'s is a \emph{contour} of $\omega.$ Contours are connected sets of $(d-1)$-dimensional faces of the cubes $\mathcal{S}_i$. We denote by $| \gamma_i|$ the number of such faces in $\gamma_i$.
Let $\Gamma_{\Lambda} := \{ \gamma \in \Gamma'(\omega) : \omega \in \Omega_{\Lambda}^{+} \}$ denote the set of all possible contours in $\Lambda.$ Finally, a collection of contours $\Gamma' \subset \Gamma_{\Lambda}$ is said to be \emph{admissible} if there exists a spin configuration $\omega \in \Omega_{\Lambda}^{+}$ such that $\Gamma'(\omega) =\Gamma'.$

Thus when $\Lambda$ is simply connected (which we will assume from now on in this paper), the partition function can be rewritten as
$$Z^+_{\Lambda; \beta,0} = e^{\beta |\EEE_{\Lambda}^+|} \Xi^+_{\Lambda;\beta,0} ,$$
where
$$\Xi^+_{\Lambda;\beta,0} := \sum_{\Gamma' \subset \Gamma_{\Lambda} admissible} \prod_{\gamma \in \Gamma'} e^{-2 \beta |\gamma|}.$$

The cluster expansion is an expression for $\log \Xi^+_{\Lambda;\beta,0}$ as an absolutely convergent series.
In this case, a cluster is a collection $X= \{\gamma_1, \dots, \gamma_r\}$ of contours such that for every two contours $\gamma_i$ and $\gamma_j$, there is a ``path'' of faces of contours of $X$ connecting $\gamma_i$ and $\gamma_j$. Note that $X$ is actually a multiset, and denote by $n_X(\gamma)$ the number of copies of $\gamma$ appearing in $X$ and by $\overline{X}$ the support of $X$, ie $\overline{X}= \cup_{\gamma \in X} \gamma.$
In the following we write $\overline{X} \subseteq \Lambda$ to say that $\overline{X} \subseteq \cup_{i \in \Lambda} \mathcal{S}_i$ as subsets of $\R^d$.

In can be shown (see e.g. \cite[Chapter 5]{Velenik}) that the cluster expansion converges for $\beta$ large enough.

\begin{proposition}
\label{prop:cluesterexp}
There exists $\beta_{\text{ce}}^{\text{lt}}(d)$ such that for all $\beta > \beta_{\text{ce}}^{\text{lt}}(d)$,
$$\log \Xi^+_{\Lambda;\beta,0} = \sum_{X: \overline{X} \subseteq \Lambda} \Psi_{\beta}(X),$$
where for a cluster $X= \{\gamma_1, \dots, \gamma_r\}$,
$$\Psi_\beta(X):= \left( \prod_{\gamma \in \Gamma_{\Lambda}} \frac{1}{n_X(\gamma)!} \right) \phi(\gamma_1,\dots,\gamma_r)  e^{-2 \beta \sum_{i=1}^r |\gamma_i|},$$
where
$$\phi(\gamma_1,\dots,\gamma_r)=\sum_{G \subseteq G_r connected} \prod_{\{i,j\} \in G} \zeta(\gamma_i,\gamma_j),$$
$$\zeta(\gamma_i,\gamma_j) := \begin{cases}
0 \text{ if } \gamma_i \cap \gamma_j = \emptyset,\\
-1 \text{ otherwise,}
\end{cases}$$
and $G_r$ denotes the complete graph on $r$ vertices.

\end{proposition}

\subsubsection{Bounds on joint cumulants}
\label{SectBoundLT}
This cluster expansion can be used to compute expectations and therefore deduce some bounds on joint cumulants.

Let $A \subseteq \Lambda$ and let us define $\sigma_A := \prod_{i \in A} \sigma_i$. 
Its expectation is given by
$$\langle \sigma_A \rangle^+_{\Lambda;\beta,0}= \sum_{\omega \in \Omega^+_{\Lambda}} \sigma_A(\omega) \frac{e^{-H_{\Lambda;\beta,0}(\omega)}}{Z^+_{\Lambda; \beta,0}}.$$
For any spin configuration $\omega \in \Omega^+_{\Lambda}$ and any contour  $\gamma \in \Gamma'(\omega)$, let us define the \emph{interior} of $\gamma$ (written $\Int(\gamma)$) as the set of points of $\Lambda$ which would have spin $-1$ if $\gamma$ was the only contour of $\omega$. We also write $\Int(X):=  \bigcup_{\gamma \in X} \Int(\gamma).$ Thus for any $\omega \in \Omega^+_{\Lambda}$ and any $i \in \Lambda$,
$$\sigma_i(\omega) = (-1)^{| \{\gamma \in \Gamma'(\omega) : i \in \Int(\gamma)\}|},$$
and thus
\begin{align*}
\sigma_A(\omega) &= (-1)^{\sum_{i \in A} | \{\gamma \in \Gamma'(\omega) : i \in \Int(\gamma)\}|}
\\&= (-1)^{\sum_{\gamma \in \Gamma'(\omega)} | \{i \in A : i \in \Int(\gamma)\}|}.
\end{align*}
Therefore one can write
$$\langle \sigma_A\rangle^+_{\Lambda;\beta,0}= \frac{\Xi^{+,A}_{\Lambda;\beta,0}}{\Xi^+_{\Lambda;\beta,0}},$$
where
$$\Xi^{+,A}_{\Lambda;\beta,0} := \sum_{\Gamma' \subseteq \Gamma_{\Lambda} admissible} \prod_{\gamma \in \Gamma'} (-1)^{| \{i \in A : i \in \Int(\gamma)\}|} e^{-2 \beta |\gamma|}.$$
The cluster expansion converges, which means that we have an analogue of Proposition \ref{prop:cluesterexp} for $\Xi^{+,A}_{\Lambda;\beta,0}$, and thus $\langle \sigma_A \rangle^+_{\Lambda;\beta,0}$ can be expressed as
$$\langle \sigma_A\rangle^+_{\Lambda;\beta,0} = \exp \left( \sum_{X: \overline{X} \subseteq \Lambda} \Psi^A_{\beta}(X) -  \sum_{X: \overline{X} \subseteq \Lambda} \Psi_{\beta}(X) \right),$$
where for a cluster $X= \{\gamma_1, \dots, \gamma_r\}$,

\begin{align*}
\Psi^A_\beta(X):=& \left( \prod_{\gamma \in \Gamma_{\Lambda}} \frac{1}{n_X(\gamma)!} \right) \left( \sum_{G \subseteq G_r connected} \zeta(\gamma_i,\gamma_j) \right)
\\& \times (-1)^{ \sum_{j=1}^r | \{i \in A : i \in \Int(\gamma_j)\}|} e^{-2 \beta \sum_{j=1}^r |\gamma_j|}.
\end{align*}

For all clusters $X$ such that no vertex of $A$ is in the interior of any of its contours, the exponent of $(-1)$ in the definition of $\Psi^A_\beta(X)$ is always $0$ and thus $\Psi^A_\beta(X)=\Psi_\beta(X).$
Therefore we obtain
$$\langle \sigma_A\rangle^+_{\Lambda;\beta,0} = \exp \left( \sum_{X \sim A : \overline{X} \subseteq \Lambda} (\Psi^A_{\beta}(X) -  \Psi_{\beta}(X)) \right),$$
where $X \sim A$ means that $X$ contains at least one contour $\gamma$ such that a point of $A$ is in the interior of $\gamma$.
The series is absolutely convergent and we can let $\Lambda \uparrow \Z^d$, obtaining the following proposition.

\begin{proposition}[Equation (5.49) in~\cite{Velenik}]
\label{prop:expectation}
For $\beta$ large enough,
$$\langle \sigma_A\rangle^+_{\beta,0} = \exp \left( \sum_{X \sim A } (\Psi^A_{\beta}(X) -  \Psi_{\beta}(X)) \right).$$
\end{proposition}

We now want to find estimates on the joint cumulants of the variables $\{\sigma_i : i \in A\}$, for all $A \subset \Z^d$ finite.
But in this case, it is easier to estimate first another quantity related to cumulants. We define, for some set $B$ and random variables $(Y_i)_{i \in B}$ defined on the same probability space,
$$Q \left(Y_j; j \in B \right) := \prod_{ \delta \subseteq B \atop \delta \neq \emptyset}  \ \left\langle \prod_{j \in \delta} Yj \right\rangle^{(-1)^{|\delta|}} .$$

For example,
$$Q_{\beta,0}(\sigma_1, \sigma_2) = \frac{\langle \sigma_1 \sigma_2 \rangle_{\beta,0}^+ }{\langle\sigma_1\rangle_{\beta,0}^+ \langle\sigma_2\rangle_{\beta,0}^+}.$$

We show a bound on the quantities $Q_{\beta,0} \left(\sigma_j; j \in A \right)$ for all finite $A \subset \Z^d.$ 

\begin{lemma}
\label{lem:quotient}
Let $A$ be a finite subset of $\Z^d$ of size $r$. Then for $\beta$ large enough,
$$ \left| Q_{\beta,0} \left(\sigma_j; j \in A \right)  - 1 \right| \leq C_r e^{-c \beta \ell_T(A)},$$
where $c=c(d)$ and $C_r$ are positive constants depending respectively on $d$ and $r$.
\end{lemma}

\begin{proof}
Using Proposition~\ref{prop:expectation}, we have
\[
\log Q \left(\sigma_j; j \in A \right)  = \sum_{\delta \subseteq A, \delta \neq \emptyset} (-1)^{|\delta|} \sum_{X \sim \delta} (\Psi^{\delta}_{\beta}(X) -  \Psi_{\beta}(X)).\]
Recall that $X \sim \delta$ means that at least one point of $\delta$ is in the interior of a contour in $X$.
We split the second sum depending on the exact subset $I \subseteq \delta$ of points that are in the interior of a contour in $X$.
By definition of $\Psi$ observe, if $I$ is as above, $\Psi^{\delta}_{\beta}(X)= \Psi^{I}_{\beta}(X)$.
Therefore
\begin{align*}
\log Q \left(\sigma_j; j \in A \right) &=  \sum_{\delta \subseteq A, \delta \neq \emptyset} (-1)^{|\delta|} \sum_{ I \subseteq \delta : \atop I \neq \emptyset} \sum_{X : \atop \Int(X) \cap \delta =I } (\Psi^{I}_{\beta}(X) -  \Psi_{\beta}(X))
\\ &= \sum_{ I \subseteq A : \atop I \neq \emptyset} \sum_{X : \atop I \subseteq \Int(X)} (\Psi^{I}_{\beta}(X) -  \Psi_{\beta}(X)) \sum_{\delta : \atop I \subseteq \delta \subseteq (I \cup  (A \setminus \Int(X)))} (-1)^{|\delta|}.
\end{align*}
But the last sum is equal to $0$ unless $A$ is contained in $\Int(X)$. Therefore we obtain
\begin{align*}
\log Q  \left(\sigma_j; j \in A \right)  &= \sum_{X : \atop A \subseteq \Int(X)} \sum_{ I \subseteq A : \atop I \neq \emptyset} (-1)^{|I|}(\Psi^{I}_{\beta}(X) -  \Psi_{\beta}(X)).
\end{align*}
Finally, there are $2^r-1$ non-empty subsets of $A$, and $|\Psi^{I}_{\beta}(X)| \leq |\Psi_{\beta}(X)|$ for all $X$ and $I$, thus
\begin{equation}
\label{eq:boundlogQ}
\left| \log Q \left(\sigma_j; j \in A \right) \right| \leq \sum_{X : \atop A \subseteq \Int(X)} 2(2^r-1) |\Psi_{\beta}(X)|.
\end{equation}
We conclude by using a trick similar to Lemma~\ref{lem:BoudingTheSum_HT}.
By~\cite[Equation (5.31)]{Velenik}, if $\beta \geq \beta_{\text{ce}}^{\text{lt}}(d)$, for any vertex $v \in \Z^d$, we have the bound
$$\sum_{X: \overline{X} \ni v} \left| \Psi_{\beta} (X)\right| \leq 1.$$
Thus if $\beta \geq 2 \beta_{\text{ce}}^{\text{lt}}(d)$,
$$ \sum_{X: \overline{X} \ni v} \left| \Psi_{\beta}(X) e^{\beta |\overline{X}|}\right| \leq \sum_{X: \overline{X} \ni v} \left| \Psi_{\frac{\beta}{2}} (X)\right| \leq 1.$$
So for any positive integer $R$,
\begin{equation}
\label{eq:boundbetaLT}
\sum_{X: \overline{X} \ni v \atop |\overline{X}| \geq R} \left| \Psi_{\beta}(X)\right| \leq  e^{-\beta R} \sum_{X: \overline{X} \ni v} \left| \Psi_{\beta}(X) e^{\beta |\overline{X}|}\right| \leq e^{-\beta R}.
\end{equation}
Let us now turn back to Eq.~\eqref{eq:boundlogQ}. Every cluster $X$ such that $A \subseteq \Int(X)$ satisfies $|\overline{X}| \geq 2 \ell_T(A).$ Moreover, if a cluster of size $R$ has $j_1 \in A$ in its interior, then it contains at least a point $v$ which is at distance at most $R$ of $j_1$. There are at most $C R^d$ such points, therefore
\begin{align*}
\sum_{X : \atop A \subseteq \Int(X)} |\Psi_{\beta}(X)| &\leq \sum_{R \geq 2\ell_T(A)} C R^d  \left[\sum_{X: \overline{X} \ni v \atop |\overline{X}| = R} \left| \Psi_{\beta}(X)\right| \right] \\
&\leq C \sum_{R \geq 2\ell_T(A)} R^d e^{-\beta R} \\
&\leq C' e^{-c \beta \ell_T(A)},
\end{align*}
for $\beta$ large enough, where $C'$ and $c$ are some positive constants
(and $c$ depends on the dimension $d$ of the ambiant space).
Thus by~\eqref{eq:boundlogQ},
$$\left| \log Q \left(\sigma_j; j \in A \right) \right| \leq C'_r e^{-c \beta \ell_T(A)},$$
for some positive constant $C'_r$ depending on $r$.
Exponentiating completes the proof.
\end{proof}

Now we can convert this estimate into the desired bound on joint cumulants.

\begin{proposition}
\label{prop:boundcum}
Let $A$ be a finite subset of $\Z^d$ of size $r$. Then, for $\beta$ large enough,
$$|\kappa_{\beta,0} \left(\sigma_j; j \in A \right)| 
\leq D_r e^{\frac{-c \beta \ell_T(A)}{2}},$$
where $c=c(d)$ is given by \cref{lem:quotient}
and $D_r$ is a positive constant depending on $r$.
\end{proposition}
\begin{proof}
By Lemma~\ref{lem:quotient},
$$ Q_{\beta,0} \left(\sigma_j; j \in A \right) = 1+ \O(\MWST{G[A]}),$$
where $G$ is the weighted graph defined on $\Z^d$ such that for each $e=(i,j)$, $w_e = e^{\frac{-c \beta \dist{i}{j}}{2}}$. Indeed in that case $\MWST{G[A]} = e^{\frac{-c \beta \ell_T'(A)}{2}} \geq e^{-c \beta \ell_T(A)}$; see the discussion in \cref{Subsect:Spanning}.

Then using Proposition 5.8 of~\cite{Valentin}, we deduce that
$$|\kappa_{\beta,0} \left(\sigma_j; j \in A \right)| = \prod_{j \in A} \langle\sigma_j\rangle_{\beta,0}^+ \times \O(\MWST{G[A]}).$$
But for all $j \in A$, we have
$\langle \sigma_j \rangle_{\beta,0}^+ \leq 1,$
and  for $\beta$ large enough, $\langle \sigma_j \rangle_{\beta,0}^+ > 0.$
Hence
$$0 < \prod_{j \in A} \langle \sigma_j \rangle_{\beta,0}^+ \leq 1.$$
Moreover by Proposition~\ref{prop:2longueurs},
$$\MWST{G[A]} = e^{\frac{-c_d \beta \ell_T'(A)}{2}} \leq e^{\frac{-c_d \beta \ell_T(A)}{2}}.$$
Thus
$|\kappa_{\beta,0} \left( \sigma_j; j \in A \right)| \leq D_r  e^{\frac{-c_d \beta \ell_T(A)}{2}}$, as claimed.
\end{proof}

\subsection{With a strong magnetic field}
\label{SectBoundMF}
The last regime we consider is the Ising model with a strong magnetic field,
{\em i.e.} $h$ is bigger than some value $h_1>0$ ($h_1$ is to be determined later).
The case of negative $h$ (smaller than $-h_1<0$) is obviously symmetric.

In this regime, there is also a well-known cluster expansion for the partition function \cite[Section 5.7]{Velenik}.
Let us present it briefly.

Fix $\Lambda \subset \Z^d$ and consider the Ising model on $\Lambda$ with $+$ boundary conditions.
We first write its partition function in a suitable form.
For a subset $\Lambda^-$ of $\Lambda$, we denote 
\[\delta_e \Lambda^- = \{ \{i,j\} \in \EEE_\Lambda^b,\ i \in \Lambda^-,\, j \notin \Lambda^-\}.\]
Define also $\wt(\Lambda^-) = \exp(-2 \beta\, |\delta_e \Lambda^-| -2h|\Lambda^-|)$.
Then we have:
\begin{lemma}[strong magnetic field representation]
  With the above notation,
  \[Z^+_{\Lambda,\beta,h} = \exp \big( \beta|\EEE_\Lambda^b| +h |\Lambda| \big) 
  \left( \sum_{\Lambda^- \subseteq \Lambda} \wt(\Lambda^-) \right).\]
  \label{lem:Z_MF}
\end{lemma}
\begin{proof}
  The proof is not difficult and can be found, {\em e.g.}, in \cite[Section 5.7]{Velenik}.
  It is important to note that the sum over $\Lambda^- \subseteq \Lambda$ corresponds to
  the sum over spin configurations in the definition of the partition function:
  the correspondence simply associates with a spin configuration
  the set $\Lambda^-$ of positions of its minus spins.
\end{proof}
Let $A$ be a subset of $\Lambda$.
It is straightforward to modify the argument to find a similar expression for the numerators
of $\langle \sigma_A \rangle$ (as in \cref{Sect:CE_LT}) or of $\left\langle \exp\left(\sum_{i \in A} t_i \sigma_i \right) \right\rangle$ (as in \cref{Sect:CE_HT}).
\begin{align}
  \sum_{\omega \in \Omega_\Lambda} \sigma_A(\omega) e^{-H^+_{\Lambda;\beta,h}(\omega)} &= \exp \big( \beta|\EEE_\Lambda^b| +h |\Lambda| \big)              
  \left( \sum_{\Lambda^- \subseteq \Lambda} (-1)^{|A \cap \Lambda^-|} \wt(\Lambda^-) \right).\label{eq:SA_MF}\\
    \sum_{\omega \in \Omega_\Lambda} \exp\left(\sum_{i \in A} t_i \sigma_i \right) &= 
e   \exp \left( \beta|\EEE_\Lambda^b| +h |\Lambda| + \sum_{i \in A} t_i \right)
    \, \left( \sum_{\Lambda^- \subseteq \Lambda} \left[\prod_{i \in A \cap \Lambda^-} \exp(-2t_i)\right] \wt(\Lambda^-) \right).
    \label{eq:MGF_MF}
\end{align}
A set $\Lambda^- \subseteq \Lambda$ can be seen as a subgraph of the lattice $\Z^d$.
As such, it admits a unique decomposition as disjoint union of its connected components $\Lambda^-=S_1 \sqcup S_2 \sqcup \dots \sqcup S_r$.
The weight $\wt$ behaves multiplicatively with respect to this decomposition $\wt(\Lambda^-)=\prod_{i=1}^r \wt(S_i)$;
the same is true for the modified weights $(-1)^{|A \cap \Lambda^-|} \wt(\Lambda^-)$ and $\left[\prod_{i \in A \cap \Lambda^-} \exp(-2t_i)\right] \wt(\Lambda^-)$
which appear in \cref{eq:SA_MF,eq:MGF_MF} above.
This enables us to use the technique of cluster expansion.

The convergence of this cluster expansion is proved for the partition function in \cite[Section 5.7]{Velenik}.
The argument can be directly adapted to get a cluster expansion of the expression in \cref{eq:SA_MF,eq:MGF_MF} above.
The same reasoning as in \cref{Sect:CE_HT} or in \cref{Sect:CE_LT} leads
to similar bounds on joint cumulants,
which proves \cref{thm:bound_joint_cumulants} in the strong magnetic field regime.

\section{Weighted dependency graphs and central limit theorems}
\label{sec:clt}
We will now use the bounds on cumulants obtained in the previous section to show that the family of random variables $\{\sigma_i : i \in \Z^d \}$ has a weighted dependency graph, and we will use this fact to deduce central limit theorems. We consider any of the regimes studied in the previous section:
very high temperature, or very low temperature with $+$ boundary condition, or strong magnetic field with any boundary condition.
To have uniform notation, we omit from now on the notation of the boundary condition in low temperature.

\subsection{Weighted dependency graph for the $\sigma_i$'s and central limit theorem for the magnetization}
\subsubsection{The weighted dependency graph}
\label{Sect:WDG_Spins}
We start by proving Theorem~\ref{th:depgraph}, which gives a weighted dependency graph for  $\{\sigma_i : i \in \Z^d \}$.

\begin{proof}[Proof of Theorem~\ref{th:depgraph}]
Let $B = \{i_1, ..., i_r \}$ be a multiset of elements of $\Z^d$ and consider the induced subgraph $\WDep[B].$
Then the maximum weight $\mathcal{M}(\WDep[B])$ of a spanning tree in $\WDep[B]$ satisfies
$$ \mathcal{M}(\WDep[B])= \eps^{\frac{\ell'_T(B)}{2}}.$$
Thus by Proposition~\ref{prop:2longueurs},
\begin{equation}
\label{eq:boundsM}
\eps^{\ell_T(B)} \leq \mathcal{M}(\WDep[B]) \leq \eps^{\frac{\ell_T(B)}{2}}.
\end{equation} 

By Proposition 5.2 of \cite{Valentin}, it is sufficient to show that
$$\left| \kappa_{\beta,h} \left( \prod_{\alpha \in B_1} \sigma_{\alpha}, \dots, \prod_{\alpha \in B_k} \sigma_{\alpha} \right) \right| \leq D_r \mathcal{M}(\WDep[B]),$$
for some sequence $\mathbf{D}= (D_r)_{r \geq 1}$, where $B_1, \dots , B_k$ are the vertex-set of the connected components of $\WDep_1[B]$, which is the graph induced by edges of weight $1$ of $\WDep$ on $B$.

The vertices $i$ and $j$ are connected in $\WDep_1$ if and only if $i=j$, because of the definition of the weights $w_e$ in $\WDep.$ Moreover the random variables $\sigma_{i}$ are equal to $+1$ or $-1$, thus
$$\sigma_i^j = \begin{cases}
\sigma_i \text{ if $j$ odd,}\\
1 \text{ if $j$ even.}
\end{cases}$$
Therefore it is sufficient to prove that for any set $B'$ of distinct $i_1, \dots, i_r$,
\begin{equation}
\label{eq:condprop5.2}
\left| \kappa_{\beta,h} \left( \sigma_{i_1}, \dots, \sigma_{i_r} \right) \right| \leq D_r \mathcal{M}(\WDep[B']).
\end{equation}

But by Theorem~\ref{thm:bound_joint_cumulants},
$$\left| \kappa_{\beta,h} \left( \sigma_{i_1}, \dots, \sigma_{i_r} \right) \right| \leq D_r \eps^{\ell_T(B')},$$
for some sequence $\mathbf{D}$ depending only on $r$. Thus using~\eqref{eq:boundsM}, Equation~\eqref{eq:condprop5.2} is proved, which completes the proof of the theorem.
\end{proof}

\begin{remark}
For all $i \in \Z^d$, let us define $X_i = \frac{\sigma_i+1}{2}$. Thus $X_i = 1$ (resp. $0$) if and only if $\sigma_i=1$ (resp. $-1$). In the remainder of this paper, it will sometimes be more convenient to consider the $X_i$'s rather than the $\sigma_i$'s. 
Since $\kappa_{\beta,h}(X_{i_1},\dots,X_{i_r})=\tfrac{1}{2^r} \kappa_{\beta,h}(\si_1,\dots,\si_r)$ for $r\ge 2$,
the weighted graph $\WDep$ defined in Theorem~\ref{th:depgraph} is also a $\mathbf{C}$-weighted dependency graph for the family $\{X_i : i \in \Z^d \}$,
for some sequence  $\mathbf{C} = (C_r)_{r \geq 1}$.
\end{remark}

\subsubsection{The central limit theorem}
We now use the weighted dependency graph from last section to obtain the central limit theorem for the magnetization.

We consider the Ising model on $\Z^d$, with inverse temperature $\beta$ and magnetic field $h$. For any positive integer $n$, we define $\Lambda_{n} := [-n,n]^d$ the $d$-dimensional cube centred at $0$ of side $2n$.
We define the \emph{magnetization}
$$S_{n} := \sum_{i \in \Lambda_n} \sigma_i,$$
and let $v_n^{2}$  denote the variance of $S_n$. Let us further define the covariance
$$\langle\sigma_i;\sigma_j\rangle_{\beta,h} := \langle\sigma_i \sigma_j\rangle_{\beta,h} -\langle\sigma_i\rangle_{\beta,h} \langle\sigma_j\rangle_{\beta,h}.$$

We now reprove the following well-known central limit for the magnetization $S_n$
(for an early reference, see \cite{Newman80}).
This serves as a warm up, to illustrate the method of dependency graphs;
moreover, some computation made in this proof will be re-used in the next section,
when studying patterns.
\begin{theorem}
\label{th:cltmagnet}
Consider the Ising model on $\Z^d$, with inverse temperature $\beta$ and magnetic field $h$, such that either $h > h_1(d)$ or $(h=0;\, \beta < \beta_1(d))$ or $(h=0;\, \beta > \beta_2(d))$.
Then, for some $v$,
$$\frac{S_n - \mathbb{E}(S_n) }{\sqrt{|\Lambda_n|}} \xrightarrow[n\to\infty]{d} \mathcal{N}(0, v^{2}).$$
Moreover $v^2 >0$ so the Gaussian law is non-degenerate.
\end{theorem}

We start by a lemma of~\cite{Ellis} on the asymptotics of the variance of $S_n$.
\begin{lemma} \cite[Lemma V.7.1]{Ellis}
\label{lem:ellis}
The limit
$$v^{2} := \lim_{n \rightarrow \infty} \frac{v_n^{2}}{|\Lambda_n|}$$
exists as an extended real valued number and
$$v^{2} = \sum_{i \in \Z^d} \langle\sigma_0;\sigma_i\rangle_{\beta,h}.$$
\end{lemma}

But by Theorem~\ref{thm:bound_joint_cumulants}, in the regimes we consider,
the cumulants (so in particular the covariance) are exponentially small, so the sum is absolutely convergent and we actually have the stronger statement:

\begin{corollary}
\label{cor:ellisfini}
Suppose that either $h > h_1$ or ($h=0$ and $\beta < \beta_1$) or ($h=0$ and $\beta > \beta_2$).
The limit
$$v^{2} := \lim_{n \rightarrow \infty} \frac{v_n^{2}}{|\Lambda_n|} =\sum_{i \in \Z^d} \langle\sigma_0;\sigma_i\rangle_{\beta,h}$$
is finite.
\end{corollary}

\begin{proof}[Proof of Theorem~\ref{th:cltmagnet}]
We will use Theorem~\ref{th:4.11modif}.
Let $\WDep$ be the weighted dependency graph defined in Theorem~\ref{th:depgraph}. Then for all $n$, $\WDep[\Lambda_n]$ is a $\mathbf{C}$-weighted dependency graph for $\{\sigma_i; i \in \Lambda_n\}$.
The number of vertices of $\WDep[\Lambda_n]$ is
$$N_n = \left| \Lambda_n \right| = (2n+1)^d,$$
and its maximal weighted degree is
$$\Delta_n -1 = \max_{i \in \Lambda_n} \sum_{j \in \Lambda_n} \eps^{\frac{\dist{i}{j}}{2}}.$$

There are $2^d {d+y-1 \choose d-1}$ points at distance $y$ of $0$ in $\Z^d$. Indeed such a point has coordinates $(y_1, \dots, y_d)$ such that $|y_1| + \cdots + |y_d|=y.$ There are ${d+y-1 \choose d-1}$ choices for the values of $|y_1|, \dots, |y_d|$, and each $y_i$ can be either positive or negative, which multiplies the number of choices by $2^d.$
Thus there are at most $2^d {d+y-1 \choose d-1}$ points at distance $y$ of any point $x$ in $\Lambda_n$, and
\[
\Delta_n -1 \leq \sum_{y=0}^{2dn} \eps^{\frac{y}{2}} 2^d {d+y-1 \choose d-1}
 \leq C,
 \]
for some constant $C$ because the infinite series is absolutely convergent.

We now have to find a sequence $(a_n)$ and integers $s$ and $v$ such that conditions (1)-(3) of Theorem~\ref{th:4.11modif} are satisfied.
We set for all n, $a_n= \sqrt{|\Lambda_n|}=(2n+1)^{\frac{d}{2}},$ $v= \sqrt{ \sum_{i \in \Z^d} \langle\sigma_0;\sigma_i\rangle_{\beta,h}}$ as in Lemma~\ref{lem:ellis} and we can choose $s$ to be any integer $\geq 3$.

Now condition (1) is satisfied because of Lemma~\ref{lem:ellis}, as
$$ \frac{v_n^{2}}{|\Lambda_n|} = \frac{v_n^{2}}{a_n^2} \xrightarrow[n \to\infty]{} v^{2}.$$

Condition (2) is also satisfied as $a_n^2 = (2n+1)^d = N_n$.

Finally, for some constant $C'$,
$$\left( \frac{N_n}{\Delta_n} \right) ^{\frac{1}{s}} \frac{\Delta_n}{a_n} \leq C' \frac{(2n+1)^{\frac{d}{s}}}{(2n+1)^{\frac{d}{2}}},$$
and the right-hand side tends to $0$ as $n$ tends to infinity for $s \geq 3$. So (3) is satisfied too.

The central limit theorem is proved.

Moreover, whatever the values of $\beta$ and $h$ are, the spin at $0$ is not constant,
thus $\langle \sigma_0,\sigma_0 \rangle_{\beta,h} >0$. On the other hand, because of the GKS inequalities~\cite{Griffiths,KS},
for all $i \in \Z^d$,  $\langle\sigma_0;\sigma_i\rangle_{\beta,h} \geq 0$ (see eg.~\cite{Velenik}). Therefore,
$$v^{2} = \langle\sigma_0,\sigma_0\rangle_{\beta,h} + \sum_{i \in Z^d \setminus \{0\} } \langle\sigma_0;\sigma_i\rangle_{\beta,h} > 0,$$
which ends the proof of the theorem.
\end{proof}

\subsection{Central limit theorem for occurrences of given patterns}

\label{sec:wdgconfig}
\subsubsection{Power of weighted dependency graphs}
A major advantage of the theory of weighted dependency graphs
is that this structure is stable by taking {\em powers}.
\begin{definition}
  Let $\WDep$ be an edge-weighted graph with vertex set $A$ and weight function $w$;
  we also consider a positive integer $m$.
  We denote by $\mathrm{MSet}_{\leq m}(A)$ the subset of multisets of elements of $A$ with cardinality at most $m$.
  Then the $m$-th power $\WDep^m$ of $\WDep$ is by definition the graph with
  vertex-set $\mathrm{MSet}_{\leq m}(A)$ and where the weight between $I$ and $J$ is given by
  $w_m(I,J)=\max_{i \in I,j \in J} w(i,j)$.
(Edges not in the graph should be seen as edges of weight $0$.)
\end{definition}
This definition is justified by the following property, proved in \cite[Section 5.3]{Valentin}.
\begin{proposition}\label{Prop:powers}
  Let $\{Y_a, a \in A\}$ be a family of random variables with a weighted dependency graph $\WDep$.
  Then $\WDep^m$ is a weighted dependency graph for the family 
  $\{Y_I, I \in \mathrm{MSet_{\leq m}(A)}\}$, where $Y_I = \prod_{a \in I} Y_a$.
\end{proposition}

Instead of applying this to the variables $\sigma_i$,
we will rather work with the variables $X_{(i,+)}:=X_i=\tfrac{1+\sigma_i}{2}$ and $X_{(i,-)}=1-X_i$. 
Start with the following observation. We have, for all $A = \{i_1, \dots, i_r\} \subseteq \Z^d$,
$$\left| \kappa_{\beta,h} (X_{i_1}, \dots, X_{i_r}) \right| = \left| \kappa_{\beta,h} (Y_{i_1}, \dots, Y_{i_r}) \right|,$$
where on some subset $B$ of $A$, $Y_i = X_{(i,+)}$, and on $A \setminus B$, $Y_i = X_{(i,-)}.$

Mimicking the proof of Theorem~\ref{th:depgraph}, we obtain the following.
Let $\WDep_s$ be the complete weighted graph with vertex set $\Z^d \times \{+,-\}$, such that for all $i,j \in \Z^d$,
\[
w'((i,+),(j,+))=w'((i,+),(j,-))= \eps^{\frac12 \dist{i}{j}}.
\]
In other words, we ignore the sign and use the weight function $w$ from the previous section.
Then $\WDep_s$ is a $\mathbf{C}$-weighted dependency graph for the family 
$\{X_{(i,s)}; i \in \Z^d, s \in \{+,-\} \}$, for some sequence  $\mathbf{C} = (C_r)_{r \geq 1}$.

By considering the powers of $\WDep_s$ and using \cref{Prop:powers},
we obtain weighted dependency graphs for the products of $X_{(i,+)}$'s and $X_{(i,-)}$'s with a bounded number of terms.
\begin{theorem}
Consider the Ising model on $\Z^d$, with inverse temperature $\beta$ and magnetic field $h$, either 
for $h > h_1$ or ($h=0$ and $\beta < \beta_1$) or ($h=0$ and $\beta > \beta_2$). 
Let $m$ be a fixed positive integer; 
for multisets $I$ of elements of $\Z^d \times \{+,-\}$, we define
$Z_I := \prod_{i \in I} X_i$.
Then $\WDep_s^m$ is a $\mathbf{D}_m$-weighted dependency graph for the family of random variables
$\{Z_I;I \in \mathrm{MSet_{\leq m}\left(\Z^d \times \{+,-\}\right)}\}$, for some sequence  $\mathbf{D}_m$ depending only on $m$.
  \label{th:powers}
\end{theorem}

\subsubsection{Local patterns}
In this section, we prove Theorem~\ref{th:cltpower}, the CLT for the number of occurrences of a given local pattern
of spins (for example isolated $+$ spins).

To find a weighted dependency graph for the potential occurrences
of a pattern $\mathcal{P}$ of size $m$, we consider $\WDep_{\mathcal{P}}$ the restriction of $\WDep_s^m$ to the $Z_I$'s of the form
\begin{equation}
  Z_i^{\mathcal{P}}= \prod_{j \in \mathcal{D}} X_{(i+j,\mathfrak{s}(j))}.
  \label{eq:ZiP}
\end{equation}
Note that vertices of $\WDep_{\mathcal{P}}$ are canonically indexed by $i \in \Z^d$
so that we will think of $\WDep_{\mathcal{P}}$ as a graph with vertex set $\Z^d$.
The weight of the edge between $i_1$ and $i_2$ is then
\[w_{\mathcal{P}}(i_1,i_2)=\max_{j_1 \in \mathcal{D},j_2 \in \mathcal{D}} w(i_1+j_1,i_2+j_2) \le 
\eps^{\frac{1}{2} \left(\dist{i_1}{i_2} - \max_{\alpha,\beta \in \mathcal{D}} \dist{\alpha}{\beta}\right)}.\]
The graph $\WDep_{\mathcal{P}}$ is a $\mathbf{D}$-weighted dependency graph for $\{Z_i, i \in \Z^d\}$,
for some sequence $\mathbf{D}$ depending only on $\mathcal{P}$.
Indeed, it is a restriction of the weighted dependency graph given in \cref{th:powers}.

We define
$$S_{n,\mathcal{P}} := \sum_{i \in \Lambda_n} Z_i^{\mathcal{P}},$$
the number of occurrences of $\mathcal{P}$ whose position is in $\Lambda_n$.
In the example of isolated $+$ spins, we have
$$S_{n,\mathcal{P}} := \sum_{i \in \Lambda_{n} } \left(X_{i} \prod_{j : \dist{i}{j}=1} (1-X_j) \right).$$
It is also easy to encode in this framework the number of $+$ connected components of any given shape.

Let $v_{n,\mathcal{P}}^2$ denote the variance of $S_{n,\mathcal{P}}$.
We have a lemma analogous to Lemma~\ref{lem:ellis}.
\begin{lemma}
\label{lem:ellispower}
As $n$ tends to infinity, the quantity $\frac{v_{n,\mathcal{P}}^2}{|\Lambda_n|}$ tends to $v_{\mathcal{P}}^2 := \sum_{k \in \Z^d} \langle Z_0^{\mathcal{P}} ; Z_k^{\mathcal{P}} \rangle_{\beta,h}<\infty$. 
\end{lemma}
\begin{proof}
  That $G_{\mathcal{P}}$ is a weighted dependency graph for the family $Z^\mathcal{P}_i$ implies
  that 
  \[\langle Z_0^{\mathcal{P}} ; Z_k^{\mathcal{P}} \rangle_{\beta,h} \le 
  D_2 \eps^{\frac{1}{2} \left(k - \max_{\alpha,\beta \in \mathcal{D}} \dist{\alpha}{\beta}\right)}.\] 
  This proves that $v_{\mathcal{P}}^2$ is finite as claimed.

Let $\eps >0$ be fixed. We want to show that for $n$ large enough,
$$\left| \frac{v_{n,\mathcal{P}}^2}{|\Lambda_n|} - \sum_{k \in \Z^d} \langle Z_0^{\mathcal{P}} ; Z_k^{\mathcal{P}} \rangle_{\beta,h} \right| \leq \eps.$$
We have
\begin{align*}
\frac{v_{n,\mathcal{P}}^2}{|\Lambda_n|} &= \frac{1}{|\Lambda_n|} \sum_{i \in \Lambda_n} \sum_{j \in \Lambda_n} \langle Z_i^{\mathcal{P}} ; Z_j^{\mathcal{P}} \rangle_{\beta,h}
\\&= \frac{1}{|\Lambda_n|} \sum_{i \in \Lambda_n} \sum_{j \in \Z^d} \langle Z_i^{\mathcal{P}} ; Z_j^{\mathcal{P}}\rangle_{\beta,h} - \frac{1}{|\Lambda_n|} \sum_{i \in \Lambda_n} \sum_{j \in \Z^d \setminus \Lambda_n} \langle Z_i^{\mathcal{P}} ; Z_j^{\mathcal{P}} \rangle_{\beta,h}.
\end{align*}
By translation invariance of $\langle Z_i^{\mathcal{P}} ; Z_j^{\mathcal{P}} \rangle_{\beta,h}$, the first sum equals $\sum_{k \in \Z^d} \langle Z_0^{\mathcal{P}} ; Z_k^{\mathcal{P}} \rangle_{\beta,h}$.
Thus the only thing left to do is to show that for $n$ large enough, the absolute value of the second term is bounded by $\eps$.
We cut the sum on $i$ into two parts : the points that are far from the boundary of $\Lambda_n$ and those which are not.
Recall that the boundary $\partial \Lambda_n$ consists of points $j$ not in $\Lambda$,
which have a neighbour in $\Lambda_n$.
We denote $\delta(i,\partial \Lambda_n)=\min_{j \in \partial \Lambda_n} \parallel i-j \parallel_1$,
which is the distance between $i$ and $\partial \Lambda_n$
For $R$ a positive integer,
let us consider the points $i \in \Lambda_n$ at distance more than $R$ from the boundary of $\Lambda_n$.
We have, again by translation invariance
$$\frac{1}{|\Lambda_n|} \sum_{i \in \Lambda_n \atop \delta(i, \partial \Lambda_n) >R} \sum_{j \in \Z^d \setminus \Lambda_n} |\langle Z_i^{\mathcal{P}} ; Z_j^{\mathcal{P}} \rangle_{\beta,h}| \leq \sum_{k \in Z^d \atop |k| > R}  |\langle Z_0^{\mathcal{P}} ; Z_k^{\mathcal{P}} \rangle_{\beta,h}|.$$
But the series $\sum_{k \in \Z^d} |\langle Z_0^{\mathcal{P}} ; Z_k^{\mathcal{P}} \rangle_{\beta,h}|$ is absolutely convergent so the sum above tends to $0$ as $R$ tends to infinity. Therefore, there exists some integer $R_0$ such that
\begin{equation}
\label{eq:premierepsilon}
\frac{1}{|\Lambda_n|} \sum_{i \in \Lambda_n \atop \delta(i, \partial \Lambda_n) >R_0} \sum_{j \in \Z^d \setminus \Lambda_n} |\langle Z_i^{\mathcal{P}} ; Z_j^{\mathcal{P}} \rangle_{\beta,h}| \leq \frac{\eps}{2}.
\end{equation}

Now let us consider the points of $\Lambda_n$ that are at distance at most $R_0$ of $\partial \Lambda_n$. There are at most $C |\partial \Lambda_n| R_0^d$ such points. Therefore
\[
\frac{1}{|\Lambda_n|} \sum_{i \in \Lambda_n \atop \delta(i, \partial \Lambda_n)\leq R_0} \sum_{j \in \Z^d \setminus \Lambda_n} |\langle Z_i^{\mathcal{P}} ; Z_j^{\mathcal{P}} \rangle_{\beta,h}| \leq \frac{1}{|\Lambda_n|} \sum_{i \in \Lambda_n \atop \delta(i, \partial \Lambda_n)\leq R_0} \sum_{k \in \Z^d} |\langle Z_0^{\mathcal{P}} ; Z_k^{\mathcal{P}} \rangle_{\beta,h}|
\leq C' \frac{|\partial \Lambda_n|}{|\Lambda_n|} R_0^d.
\]

But as $n$ tends to $\infty$, $\frac{|\partial \Lambda_n|}{|\Lambda_n|}$ tends to $0$. Therefore for $n$ large enough,
\begin{equation}
\label{eq:deuxiemeepsilon}
\frac{1}{|\Lambda_n|} \sum_{i \in \Lambda_n \atop \delta(i, \partial \Lambda_n)\leq R_0} \sum_{j \in \Z^d \setminus \Lambda_n} |\langle Z_i^{\mathcal{P}} ; Z_j^{\mathcal{P}} \rangle_{\beta,h}| \leq \frac{\eps}{2}.
\end{equation}

Adding~\eqref{eq:premierepsilon} and~\eqref{eq:deuxiemeepsilon} completes the proof.
\end{proof} 
 
We are now ready to prove the central limit theorem.

\begin{proof}[Proof of \cref{th:cltpower}]
We proceed as in the proof of Theorem~\ref{th:cltmagnet}.
We consider $\WDep_{\mathcal{P}}[\Lambda_n]$.

The number of vertices is $N_n = |\Lambda'_n| = |\Lambda_n|$ and,
from the discussion above, its maximal weighted degree $\Delta_n-1$ is bounded as follows:
\[
  \Delta_n -1 \le \max_{i \in \Lambda_n} \sum_{j \in \Lambda_n} \eps^{\tfrac{1}{2} \left(\dist{i}{j} - \max_{\alpha,\beta \in \mathcal{D}} \dist{\alpha}{\beta}\right)}
   \leq \max_{i \in \Lambda_n} \sum_{j \in \Lambda_n} C_{\mathcal{P}}\eps^{\tfrac{1}{2} \dist{i}{j}},
   \]
where $C_{\mathcal{P}}$ is a positive constant depending only on the pattern $\mathcal{P}.$
Thus by the same argument as in the proof of Theorem~\ref{th:cltmagnet},
$
\Delta_n -1 \leq C'_{\mathcal{P}},
$
for some other constant $C'_{\mathcal{P}}$.

Again we set for all $n$, $a_n= \sqrt{|\Lambda_n|}.$ We also set $v = v_{\mathcal{P}}$ as in Lemma~\ref{lem:ellispower} and we can choose $s$ to be any integer $\geq 3$.

Conditions (1) to (3) of Theorem~\ref{th:4.11modif} are satisfied again and the theorem is proved.
\end{proof}
\begin{remark}
  The variance $v_{\mathcal{P}}$ appearing in \cref{th:cltpower}
  might be equal to $0$ for some patterns $\mathcal{P}$,
in which case the central limit theorem is degenerate.
If the pattern has only plus spins,
the same proof as before gives $v_{\mathcal{P}}>0$.
\end{remark}

\subsubsection{Global patterns}
In this final section, we establish Theorem \ref{thm:GlobalPatterns}, the central limit theorem for the number of occurrences of a global pattern of spins.

To find a weighted dependency graph for the potential occurrences of $\GP$ of size $m$, we consider $\WDep_{\GP}$ the restriction of $\WDep_s^m$ to the $Z_I$'s of the form
\begin{equation}
  Z_{\{x^{(1)}, \dots ,x^{(m)}\}}^{\GP}= \prod_{i=1}^m X_{(x^{(i)},\mathfrak{s}(i))}.
  \label{eq:ZXGP}
\end{equation}
In $\WDep_{\GP}$, the weight of the edge between $\{x^{(1)}, \dots ,x^{(m)}\}$ and $\{y^{(1)}, \dots ,y^{(m)}\}$ is given by
\[ w_{\GP}(\{x^{(1)}, \dots ,x^{(m)}\},\{y^{(1)}, \dots ,y^{(m)}\})=\max_{i,j \in \{1, \dots , m\}} w(x^{(i)},y^{(j)})=\eps^{\frac{1}{2} \min_{i,j \in \{1, \dots , m\}}\dist{x^{(i)}}{y^{(j)}}}.\]
Again, the graph $\WDep_{\GP}$ is a $\mathbf{D}$-weighted dependency graph for $\{ Z_{\{x^{(1)}, \dots ,x^{(m)}\}}^{\GP},
\{x^{(1)}, \dots ,x^{(m)}\} \subset \Z^d\}$,
for some sequence $\mathbf{D}$ depending only on $\mathcal{P}$
as it is a restriction of the weighted dependency graph given in \cref{th:powers}.

Now define
$$S_{n,\GP} := \sum_{\{x^{(1)}, \dots ,x^{(m)}\} \subset \Lambda_n} Z_{\{x^{(1)}, \dots ,x^{(m)}\}}^{\GP},$$
the number of occurrences of $\GP$ in $\Lambda_n$.
Let $v_{n,\GP}^2$ denote the variance of $S_{n,\GP}$.

\begin{proof}
  [Proof of \cref{thm:GlobalPatterns}]
Consider the weighted dependency graph $\WDep_{\GP}[\Lambda_n]$.
Its number of vertices is $N_n^m = |\Lambda_n|^m$. Let us now bound its maximal weighted degree $\Delta_n-1$.
Fix $\{x^{(1)}, \dots ,x^{(m)}\} \subset \Lambda_n.$ We have
\begin{align*}
  \sum_{\{y^{(1)}, \dots , y^{(m)}\} \subset \Lambda_n} \eps^{\frac{1}{2} \min_{i,j \in \{1, \dots , m\}}\dist{x^{(i)}}{y^{(j)}}} &\leq \sum_{y^{(1)}, \dots , y^{(m)} \in \Lambda_n} \sum_{i=1}^m \sum_{j=1}^m \eps^{\frac{1}{2} \dist{x^{(i)}}{y^{(j)}}} \\
&\leq \sum_{i=1}^m m |\Lambda_n|^{m-1} \sum_{y \in \Lambda_n} \eps^{\frac{1}{2} \dist{x^{(i)}}{y}}.
\end{align*}
By the proof of Theorem~\ref{th:cltmagnet}, the last sum is bounded by a certain constant $C$. Thus
$$\Delta_n-1 = \max_{\{x^{(1)}, \dots ,x^{(m)}\} \subset \Lambda_n} \sum_{y^{(1)}, \dots , y^{(m)} \in \Lambda_n} \eps^{\frac{1}{2} \min_{i,j \in \{1, \dots , m\}}\dist{x^{(i)}}{y^{(j)}}} \leq m^2 |\Lambda_n|^{m-1} C.$$

We want to apply \cref{th:4.11modif} and set $a_n=\sqrt{v_{n,\GP}^2}$.
Condition (1) is trivial, while (2) holds for all weighted dependency graphs when $a_n$ is the standard deviation of $X_n$
(see \cite[Lemma 4.10]{Valentin}).
Condition (3) is fulfilled since, using \eqref{eq:Hypo_Variance} and the inequality above for $\Delta_n$,
\[\left( \frac{N_n}{\Delta_n} \right) ^{\frac{1}{s}} \frac{\Delta_n}{a_n} \le \left(\frac{|\Lambda_n|^m}{ m^2 |\Lambda_n|^{m-1} C}\right) ^{\frac{1}{s}} \frac{m^2 |\Lambda_n|^{m-1} C}{\sqrt{A |\Lambda_n|^{2m-2+\eps}}}
\le \text{cst} |\Lambda_n|^{1/s-\eps/2}\]
and the right-hand side tends to $0$ for $n$ big enough.
\end{proof}

We now show a simple sufficient condition -- the pattern consisting in positive spins only --
so that the bound \eqref{eq:Hypo_Variance} of the variance is fulfilled.

We start with a lemma.
\begin{lemma}
  \label{lem:cov_Une_Egalite}
  Fix $m \ge 2$.
  There exist some constants $R>0$ and $B>0$ such that the following holds.
  For any lists $(x^{(1)},\dots,x^{(m)})$ and $(y^{(1)},\dots,y^{(m)})$ 
  such that $x^{(1)}=y^{(1)}$ but no two elements in the set $\{x^{(1)},\dots,x^{(m)},y^{(2)},\dots,y^{(m)}\}$
  are at distance less than $R$, we have
  \[\Cov\left( \prod_{i=1}^m X_{(x^{(i)},+)}, \prod_{i=1}^m X_{(x^{(i)},+)} \right) \ge B.\]
\end{lemma}
\begin{proof}
  By definition, and using that $X_{(x^{(1)},+)} X_{(y^{(1)},+)}=X_{(x^{(1)},+)}^2=X_{(x^{(1)},+)}$, we have 
  \begin{multline*}
    \Cov\left( \prod_{i=1}^m X_{(x^{(i)},+)}, \prod_{i=1}^m X_{(x^{(i)},+)} \right) 
  =\esper\big[ X_{(x^{(1)},+)} \cdots X_{(x^{(m)},+)} X_{(y^{(2)},+)} \dots X_{(y^{(m)},+)}\big] \\
  -\esper\big[ X_{(x^{(1)},+)} \cdots X_{(x^{(m)},+)} \big] 
  \esper\big[ X_{(x^{(1)},+)}  X_{(y^{(2)},+)} \dots X_{(y^{(m)},+)}\big].
\end{multline*}
  Using the expression of joint moments in terms of cumulants -- see, e.g. \cite[Eq. (3)]{Valentin} --
  and the bound for cumulants of spins (\cref{thm:bound_joint_cumulants}),
  we have that there exists a constant $\text{cst}_m$ such that
  \[\big|\esper\big[ X_{(x^{(1)},+)} \dots X_{(x^{(m)},+)} \big] -
  \esper\big[ X_{(x^{(1)},+)}] \dots.\esper\big[ X_{(x^{(m)},+)}\big]\big|
  \le \text{cst}_m \eps^R,\]
  whenever the $x^{(i)}$ all lie at distance at least $R$ from each other.
  The same holds for the other products above and we get
  \begin{multline*}
    \Cov\left( \prod_{i=1}^m X_{(x^{(i)},+)}, \prod_{i=1}^m X_{(x^{(i)},+)} \right)
  = \big(\esper\big[ X_{(x^{(1)},+)}\big] - \esper\big[ X_{(x^{(1)},+)}\big]^2\big) \\
  \cdot
  \esper\big[ X_{(x^{(2)},+)}] \cdots.\esper\big[ X_{(x^{(m)},+)}\big]
  \esper\big[ X_{(y^{(2)},+)}] \cdots.\esper\big[ X_{(y^{(m)},+)}\big]
  +\text{error},
\end{multline*}
  where the error is uniformly bounded by $\text{cst}_m \eps^R$.
  The main term in the above equation is positive (as a product of positive terms)
  and independent from the $x^{(i)}$
  and the $y^{(i)}$ (by translation invariance), while the error can be made as small as wanted by making $R$ tend to infinity.
  This proves the lemma.
\end{proof}
\begin{proposition}
  \label{prop:Variance_Lower_Bound}
  Let $\GP$ be a global pattern of size $m$ and assume that the function $\mathfrak{s}$ defining $\GP$
  takes only value $+1$.
  Then there exists a constant $A$ such that
  $\Var(S_{n,\GP}) \ge A n^{2m-1}$.
\end{proposition}
\begin{proof}
  We expand the variance as
  \[\Var(S_{n,\GP}) = \sum_{\{x^{(1)}, \dots ,x^{(m)}\} \subset \Lambda_n \atop \{y^{(1)}, \dots ,y^{(m)}\} \subset \Lambda_n}
  \Cov(Z_{\{x^{(1)}, \dots ,x^{(m)}\}}^{\GP}, Z_{\{y^{(1)}, \dots ,y^{(m)}\}}^{\GP}).\]
  When $\GP$ involves only positive spins, the FKG inequality ensures that all summands are positive.
  Restricting the sum to sets with an ordering that fulfills
  the hypothesis of \cref{lem:cov_Une_Egalite} gives a lower bound.
  Therefore $\Var(S_{n,\GP}) \ge B \cdot N_1$, where $N_1$ is the number of pairs of sets
  $(\{x^{(1)}, \dots ,x^{(m)}\},\{y^{(1)}, \dots ,y^{(m)}\})$ as in \cref{lem:cov_Une_Egalite}.
  For fixed $R>0$, this number is clearly of order $N^{2m-1}$, finishing the proof of the proposition.
\end{proof}

\section*{Acknowledgements}
The authors are grateful to H. Duminil-Copin, R. Kotecký and D. Ueltschi
for discussions on cluster expansions and bounds on cumulants in the Ising model.

\bibliographystyle{siam}
\bibliography{biblio}

\end{document}